\documentclass[10pt,a4paper]{amsart}

\usepackage{array}
\usepackage{lipsum}
\usepackage{comment}
\usepackage{color}

\usepackage[utf8]{inputenc}
\usepackage[T1]{fontenc}
\usepackage{textcomp}
\usepackage{amsmath,amssymb,amsopn,amsthm}
\usepackage{lmodern}
\usepackage[a4paper]{geometry}
\usepackage{graphicx}
\usepackage{xcolor}
\usepackage{microtype}

\usepackage{calrsfs}

\usepackage{dsfont}

\usepackage{amsmath,amsfonts,amssymb,amsthm,mathrsfs,unicode}
\usepackage{calligra}

\usepackage{pdfpages}

\usepackage{tikz}
\usetikzlibrary{decorations.pathreplacing, arrows.meta}

\usepackage{hyperref}

\usepackage{graphicx}
\usepackage{mathtools}

\hypersetup{pdfstartview=XYZ}

\newcommand{\union}[2]{\text{$\underset{#1}{\bigcup} #2$}}

\newcommand{\intervEnt}[2]{\text{$\llbracket #1,#2\rrbracket$}}

\newcommand{\prodend}[2]{\text{${\underset{#1}{\prod}} #2$}}

\newcommand{\integrale}[4]{\text{$\int_{#1}^{#2} #3 d#4$}}

\newcommand{\norme}[1]{\text{$\left\Vert #1\right\Vert$}}

\newcommand{\abso}[1]{\text{$\left\vert #1\right\vert$}}

\newcommand{\reciproque}[1]{\text{${#1}^{-1}$}}

\newcommand{\intInd}[3]{\text{$\int_{#1} #2 d#3$}}

\newtheorem{definition}{Definition}[section]

\newtheorem{theorem}{Theorem}[section]

\newtheorem{lemme}{Lemma}[section]

\newtheorem{proposition}{Proposition}[section]

\newcounter{numberSection}[section]

\newcounter{numberSubSection}[subsection]

\addtocounter{numberSection}{1}
\addtocounter{numberSubSection}{1}

\begin{document}
\title{Keplerian shear for Chacon transformations}
\author{Arthur Boos \and Benoit Saussol}
\address{Aix-Marseille Université, Institut de Mathématiques de Marseille, CNRS, Luminy case 907, 13288 Marseille}
\date{\today}
\maketitle
\begin{abstract}
The concept of keplerian shear was introduced by Damien Thomine recently. It is useful for non ergodic systems, and can be seen as strong mixing conditionally on invariant fibers. The notion is particularly interesting when a.e. fiber is not strongly mixing.

We develop here an approach appropriate for systems such that a.e. fiber is weakly mixing, and apply it to a family of rank one  transformations. Each transformation is a kind of Chacon map, built with a random number of spacers at each step of the Rochklin tower. We prove that this new dynamical system exhibits keplerian shear. The method relies on a version of a local limit theorem for time dependent Birkhoff sums along the fullshift.
\end{abstract}
\tableofcontents
\section{Introduction}

Consider a dynamical system $(X,\mathcal{T},\mu,T)$.
We note $\mathcal{I}$ the invariant $\sigma-$algebra by the transformation $T$.
This dynamical system exhibits a keplerian shear if
\[\forall f_1,f_2\in \mathbb{L}^2_\mu(X),\mathbb{E}_\mu(\overline{f_1}f_2\circ T^n)\xrightarrow[n\to+\infty]{}\mathbb{E}_\mu\left(\overline{\mathbb{E}_\mu(f_1\vert \mathcal{I})}\mathbb{E}_\mu(f_2\vert \mathcal{I})\right).\]
In other words,\[\mathbb{E}_\mu(Cov(f_1,f_2\circ T^n\vert\mathcal{I}))\xrightarrow[n\to+\infty]{}0.\]
The keplerian shear can be seen as the strong-mixing conditionally to an invariant $\sigma-$algebra. We can see this notion for the first time in the Damien Thomine's article \cite{DaTho}.
We recall the definition of mixing system. $(X,\mathcal{T},\mu,T)$ is strong-mixing if for $A,B\in \mathcal{T}$, \[\mu(A\cap T^{-n}(B))-\mu(A)\mu(B)\xrightarrow[n\to+\infty]{}0.\]
Strong-mixing implies weak-mixing i.e
\begin{equation}\label{wm}\frac{1}{n}\sum_{j=0}^{n-1}\abso{\mu(A\cap T^{-j}(B))-\mu(A)\mu(B)}\xrightarrow[n\to+\infty]{}0.\end{equation}
More precisely, the strong-mixing is equivalent to 
\[\forall f_1,f_2\in \mathbb{L}^2_\mu(X),\mathbb{E}_\mu(\overline{f_1}f_2\circ T^n)\xrightarrow[n\to+\infty]{}\overline{\mathbb{E}_\mu(f_1)}\mathbb{E}_\mu(f_2).\]
We remark that when the system is ergodic, the invariant $\sigma-$algebra is the trivial $\sigma-$algebra $\{\emptyset,X\}$ and then, for $f\in \mathbb{L}^2_\mu(X),$ \[\mathbb{E}_\mu(f\vert\mathcal{I})=\mathbb{E}_\mu(f).\]
Thus, we deduce that if the system exhibit keplerian shear and it is ergodic, thus it is strongly mixing. Reciprocally, the strong-mixing property implies the keplerian shear and the ergodicity of the system.

\begin{proposition}\cite{Eins} $T$ is weak mixing if and only if for all  $A,B\in\mathcal{T}$, there exists $J_{A,B}\subset \mathbb{N}$ with density zero such that \[\mu(A\cap T^{-n}(B))-\mu(A)\mu(B)\xrightarrow[n\to+\infty\ and\ n\notin J_{A,B}]{}0\] 
\end{proposition} 
We call $J_{A,B}$ an exceptional set.

We emphasize that the speed of weak-mixing has been studied recently in  \cite{Avila,Moll, Solom} in the context of Interval exchange transformations, Chacon transformation and substitution $\mathbb{Z}-$actions.
In these works, the authors provide an upper bound  for the time average in \eqref{wm}. Such quantitative estimates, and the methods involved to get them, give an exceptional set from a calculus lemma. Unfortunately the only information that we can extract from this is an upper bound on the cardinality of $J_{A,B}\cap\intervEnt{0}{N}$. In particular, we have no information about where the exceptional set is located. 

In this article, considering the probabilistic space $(X,\mathcal{S},\eta)$, we construct a family of applications $T_\alpha$ such that for $\eta$-almost every $\alpha\in X$, the map $T_{\alpha}$ is not strongly mixing but weakly mixing.  
The system $(\Omega,\mathcal{S}^',\mu,T)$ fibred in $X$ with $T : (\alpha,x)\in \Omega\mapsto (\alpha,T_\alpha(x))$ and $\mu=d\nu_\alpha(x)d\eta(\alpha)$ is obviously not ergodic, still for almost every $\alpha\in X,(Y_\alpha,\mathcal{T}_\alpha,\nu_\alpha,T_\alpha)$ is weakly-mixing.
In our case, we suppose that there exists a probabilistic space $(Y,\mathcal{T},\nu)$ such that for all $\alpha\in X,Y_\alpha\subset Y$ and $\left\{T^n(A\times B) : A\in \mathcal{S},n\in\mathbb{N}\text{ and }B\in \bigcap_{\alpha\in X}\mathcal{T}_\alpha\right\}$ generate the $\sigma-$algebra $\mathcal{S}^'$.
 Then, for measurable sets $A,B\in \bigcap_{\alpha\in X}\mathcal{T}_\alpha$ and for $n\in\mathbb{N}$, we define $W_{A,B,n}=\{\alpha\in X : n\in J_{A,B}^\alpha\}.$ 
By Lebesgue dominated convergence theorem we have immediately
\begin{proposition}
Suppose that for all measurable sets $A,B$ we have $\eta(W_{A,B,n})\to0$ as $n\to\infty$. Then the system exhibits keplerian shear.
\end{proposition}
We will use this approach\footnote{For technical reasons we will use a slightly different version of this result.} to prove keplerian shear for a parametrized family rank 1 transformations, generalizing Chacon transformations, and described below. We recall that the original Chacon transformation is weakly mixing but not strongly mixing \cite{Chacon}.
The Figure \ref{rock_tour_step1} shows us the illustration of its construction at step $0$. 
\begin{figure}
\label{rock_tour_step1}
        \centering
        	\begin{tikzpicture}
		
		\draw[thick, blue] (0,0) -- (6,0);
		
		\foreach \x/\label in {0/0, 2/\frac{1}{3}, 4/\frac23, 6/1, 8/\frac43, 8.33/\frac{13}{9}, 9/\frac32} {
			\draw[thick, blue] (\x,0.1) -- (\x,-0.1) node[below, blue] {$\label$};
		}
		
		\draw[ultra thick, red] (6,0) -- (8,0);
		\draw[ultra thick, dotted, red] (8,0)--(9,0);
		
		\draw[->, cyan, thick] (1,0.3) to[out=90, in=180] (2,0.8) to[out=0, in=90] (3,0.3);
		\draw[->, cyan, thick] (3,-0.3) to[out=-90, in=180] (5,-0.8) to[out=0, in=-90] (7,-0.3);
		\draw[->, cyan, thick] (7,0.3) to[out=90, in=0] (6,0.8) to[out=-180, in=90] (5.3,0.3);
		\draw[->, cyan, thick] (5,0.3) -- (5,1);
		\node[cyan] at (2.2,4) {\textbf{?}};
		\node[cyan] at (5.2,1.1) {\textbf{?}};
		
		\node[above, red] at (7.8,0.2) {Spacers};
		
		\draw[thick, blue] (1,2) -- (3,2);
		\draw[thick, blue] (1,2.5) -- (3,2.5);
		\draw[thick, blue] (1,3.5) -- (3,3.5);
		
		\draw[ultra thick, red] (1,3) -- (3,3);
		
		\draw[->, cyan, thick] (2,2.1) -- (2,2.4);
		\draw[->, cyan, thick] (2,2.6) -- (2,2.9);
		\draw[->, cyan, thick] (2,3.1) -- (2,3.4);
		\draw[->, cyan, thick] (2,3.6) -- (2,3.9);
		
		\draw[decorate, decoration={brace, mirror, amplitude=10pt}, thick] (3.2,1.9) -- (3.2,3.6) node[midway, right, xshift=0.3cm] {$h_1 = 3+\alpha_0$ with $\alpha_0=1$};
		
	\end{tikzpicture}
        \caption{First step of the Rockhlin tower for the $d=3$-Chacon transformation}
        \label{fig:enter-label}
    \end{figure}
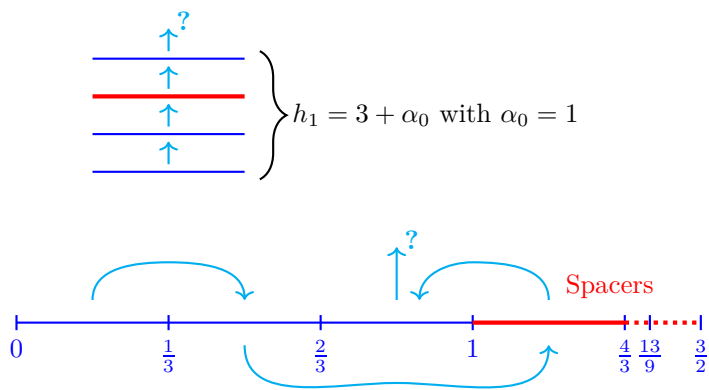

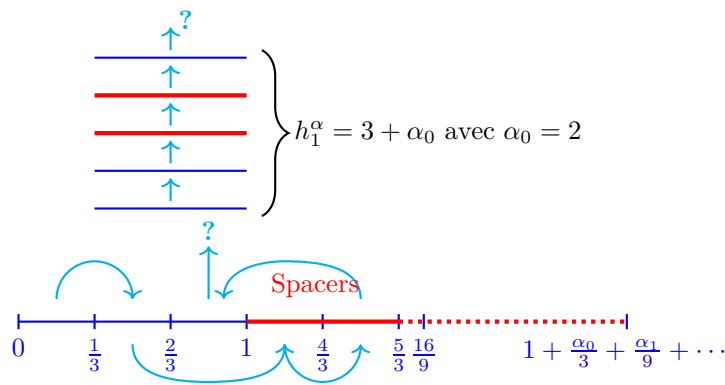
\begin{figure}
\label{rock_tour_step1}
        \centering
        	\begin{tikzpicture}
		
		\draw[thick, blue] (0,0) -- (5,0);
		
		\foreach \x/\label in {0/0, 1/\frac{1}{3}, 2/\frac23, 3/1, 4/\frac43,5/\frac53, 5.33/\frac{16}{9}, 8/1+\frac{\alpha_0}{3}+\frac{\alpha_1}{9}+\cdots} {
			\draw[thick, blue] (\x,0.1) -- (\x,-0.1) node[below, blue] {$\label$};
		}
		
		\draw[ultra thick, red] (3,0) -- (5,0);
		\draw[ultra thick, dotted, red] (5,0)--(8,0);
		
		\draw[->, cyan, thick] (0.5,0.3) to[out=90, in=180] (1,0.8) to[out=0, in=90] (1.5,0.3);
		\draw[->, cyan, thick] (1.5,-0.3) to[out=-90, in=180] (2.5,-0.8) to[out=0, in=-90] (3.5,-0.3);
		\draw[->, cyan, thick] (3.5,-0.3) to[out=-90, in=180] (4,-0.8) to[out=0, in=-90] (4.5,-0.3);
		\draw[->, cyan, thick] (4.5,0.3) to[out=90, in=0] (3.5,0.8) to[out=-180, in=90] (2.7,0.3);
		\draw[->, cyan, thick] (2.5,0.3) -- (2.5,1);
		\node[cyan] at (2.2,4) {\textbf{?}};
		\node[cyan] at (2.5,1.2) {\textbf{?}};
		
		\node[above, red] at (3.9,0.2) {Spacers};
		
		\draw[thick, blue] (1,1.5) -- (3,1.5);
		\draw[thick, blue] (1,2) -- (3,2);
		\draw[thick, blue] (1,3.5) -- (3,3.5);
		
		\draw[ultra thick, red] (1,3) -- (3,3);
		\draw[ultra thick, red] (1,2.5) -- (3,2.5);
		
		\draw[->, cyan, thick] (2,1.6) -- (2,1.9);
		\draw[->, cyan, thick] (2,2.1) -- (2,2.4);
		\draw[->, cyan, thick] (2,2.6) -- (2,2.9);
		\draw[->, cyan, thick] (2,3.1) -- (2,3.4);
		\draw[->, cyan, thick] (2,3.6) -- (2,3.9);
		
		\draw[decorate, decoration={brace, mirror, amplitude=10pt}, thick] (3.2,1.4) -- (3.2,3.6) node[midway, right, xshift=0.3cm] {$h_1^\alpha = 3+\alpha_0$ avec $\alpha_0=2$};
		
	\end{tikzpicture}
        \caption{Second step of the Rockhlin tower for $d=3$ and $\alpha_0=2$.}
        \label{fig:enter-label}
    \end{figure}

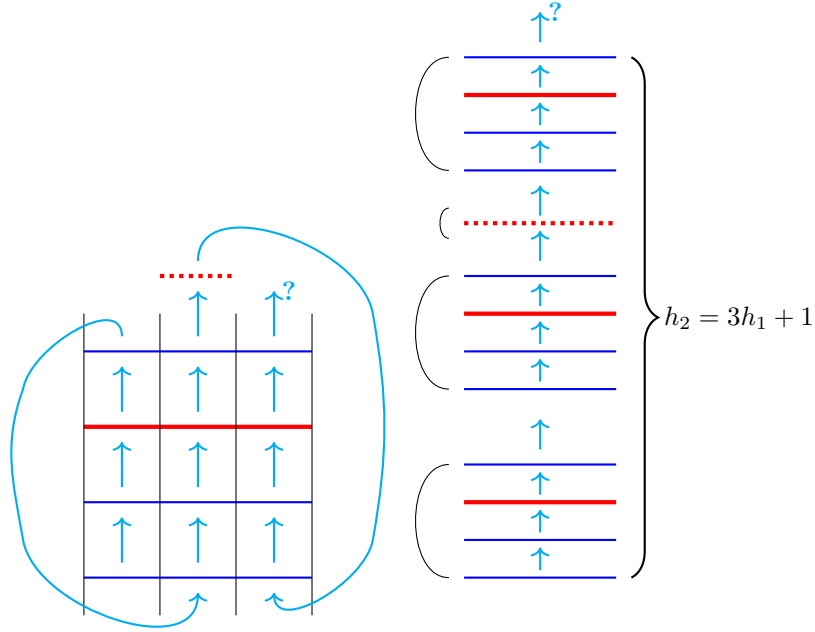
\begin{figure}
    \centering
    \begin{tikzpicture}
		\draw[thick, blue] (0,0) -- (3,0);
		\draw[thick, blue] (0,1) -- (3,1);
		\draw[ultra thick, red] (0,2) -- (3,2);
		\draw[thick, blue] (0,3) -- (3,3);
		
		\draw[black] (0,-0.5) -- (0,3.5);
		\draw[black] (1,-0.5) -- (1,3.5);
		\draw[black] (2,-0.5) -- (2,3.5);
		\draw[black] (3,-0.5) -- (3,3.5);
		
		\draw[ultra thick, dotted, red] (1,4)--(2,4);
		
		\foreach \x in {0.5,1.5,2.5} {
		 \foreach \y in {0.2,1.2,2.2} {
		 \draw[->, cyan, thick] (\x,\y) -- (\x,\y+0.6);
		 }}
	 	\draw[->, cyan, thick] (1.5,3.2) -- (1.5,3.8);
	 	\draw[->, cyan, thick] (2.5,3.2) -- (2.5,3.8); \node[cyan] at (2.7,3.8) {\textbf{?}};
	 	\draw[->, cyan, thick] (0.5,3.2) to [out=90, in=80] (-0.8,2.5) to [out=250, in=100] (-0.8,0.5) to [out=280, in=-90] (1.5,-0.2);
	 	\draw[->, cyan, thick] (1.5,4.2) to [out=90, in=100] (3.8,3.5) to [out=280, in=80] (3.8,0.5) to [out=260, in=-90] (2.5,-0.2);

	 	\foreach \a in {0,2.5,5.4} {
	 	\draw[thick, blue] (5,\a) -- (7,\a); 
	 	\draw[thick, blue] (5,\a+0.5) -- (7,\a+0.5);
	 	\draw[ultra thick, red] (5,\a+1) -- (7,\a+1);
	 	\draw[thick, blue] (5,\a+1.5) -- (7,\a+1.5);
	 	\draw[->, cyan, thick] (6,\a+0.1) -- (6,\a+0.4);
	 	\draw[->, cyan, thick] (6,\a+0.6) -- (6,\a+0.9);
	 	\draw[->, cyan, thick] (6,\a+1.1) -- (6,\a+1.4);
	 	\draw[black] (4.8,\a) to [out=180, in=180] (4.8,\a+1.5);
	 	}
	 	\foreach \a in {1.5, 4, 4.6, 6.9} {
	 		\draw[->, cyan, thick] (6,\a+0.2) -- (6,\a+0.6);
	 	}
	 	\node[cyan] at (6.2,7.5) {\textbf{?}};

	 	\draw[ultra thick, dotted, red] (5,4.7) -- (7,4.7);
	 	\draw[black] (4.8,4.5) to [out=180, in=180] (4.8,4.9);
	 	
	 	\draw[decorate, decoration={brace, mirror, amplitude=10pt}, thick] (7.2,0) -- (7.2,6.9) node[midway, right, xshift=0.3cm] {$h_2 = 3 h_1+1$};
	 	
\end{tikzpicture}
    \caption{Second step of the Rockhlin tower for the $d=3$-Chacon Transformation}
        \label{fig:enter-label}
\end{figure}

In our case, for $d\geq 7$ an odd number, we will study the $d$-Chacon transformation with $d$ intervals instead of $3$ intervals.
In \cite{Park}, for $k\in\mathbb{N}$, exceptional sets $J_{A_k,A_k}$ with\footnote{We use the notation $[a,b[$ to represent the half open interval $[a,b)$.} $A_k=\left[0,\frac{1}{d^k}\right[$ are defined explicitly, and suffice to define $J_{A,B}$ for all $A,B$. We will write $J_k$ instead of $J_{A_k,A_k}$.
Now, we introduce a sequence of random i.i.d parameters $(\alpha_j)_{j\in\mathbb{N}}\in\{1,2\}^\mathbb{N}$ with $\alpha_j$ the number of spacers at the $j-th$ step.
We write with a little abuse of language
\begin{equation}\label{abus_alpha}\alpha=\sum_{j\in\mathbb{N}}\alpha_jd^{-(j+1)}.\end{equation}
Then, let \begin{align}\label{def_set}\Omega := \union{\alpha\in \{1,2\}^\mathbb{N}}{(\{\alpha\}\times [0,1+\alpha[)}.\end{align}
We will write $\alpha=(\alpha_j)_{j\in\mathbb{N}}$ and note $T_\alpha$ the Chacon transformation associated to $\alpha$. See définition \ref{def_trans} for the precise definition.
We define now the transformation \[T : (\alpha,x)\in \Omega\mapsto (\alpha,T_\alpha(x)).\]
We consider the dynamical system $(\Omega,\mathcal{B}(\Omega),\mu ,T)$ with $\mu=d(\mathbb{L}eb_\alpha)(x)d\mathcal{H}(\alpha)$, $\mathcal{H}$ the countable product of the uniform measure on $\{1,2\}$ and $\mathbb{L}eb_\alpha$ the normalized Lebesgue measure on $[0,1+\alpha[$.
As in \cite{Park}, for each $\alpha$ and each $k\in\mathbb{N}$, we can identify an exceptional set $J_{k,\alpha}$ for $T_\alpha$. In this case, the exceptional set depends on $\alpha$.
Then, we can define the bad set of $\alpha$ for $n\in\mathbb{N}$ as\[W_{k,n} :=\{\alpha\in \{1,2\}^\mathbb{N}\ :\ n\in J_{k,\alpha}\}.\]

We will prove (see Section \ref{exception} for precise statements) that under our hypothesis 
\begin{equation}\label{first_main_prop}
    \mathcal{H}(W_{k,n})\xrightarrow[n\to+\infty]{}0.
\end{equation}

We now describe the content of the paper. 
First, in Section \ref{sec_chacon} we define precisely the parametrized family of Chacon transformations. 
In Section~\ref{sec_temps} we develop the induction procedure, and express the self correlations in terms of distribution of return time in Section~\ref{sec_return}.
In Section~\ref{exception} we define two exceptional sets for $A_k$.
Section~\ref{sec_llt} is devoted to a local limit theorem for the fiberwise distribution of return times.
Finally Section~\ref{sec_keplerian} concludes the proof of the main result.\\

\noindent{\it Acknowledgements.}
We thank Sébastien Gouezel and Thierry De La Rue for numerous suggestions and comments that improve the manuscript.

\section{Chacon transformation}\label{sec_chacon}

First of all, we define our family of Chacon transformations.

We call spacers for $m\in\mathbb{N}$, $\alpha\in\{1,2\}^\mathbb{N}$ and $q\in\intervEnt{0}{\alpha_m-1}$ the sets of the form \begin{align}\label{spacer_def}\sum_{j=0}^{m-1}\alpha_jd^{-(j+1)}+qd^{-(m+1)}+\left[0,d^{-(m+1)}\right[.\end{align}

We note for an integer $m$, \[\alpha^{(m)}:=\sum_{j=0}^m\alpha_jd^{-(j+1)}, \quad \alpha^{(-1)}=0.\]
    The Chacon transformation is the same as the original Chacon transformation with one or two spacers for each step $j$ according to the value of $\alpha_j$.
    Let note \begin{equation}\label{indice_val}Z_n=\left[1-\frac{1}{d^{n}},1-\frac{1}{d^{n+1}}\right[\cup\left[1+\alpha^{(n-1)},1+\alpha^{(n)}\right[.\end{equation}
    \begin{definition}[Chacon transformation]
    \label{def_trans}
    Let $\tau_{n,\alpha} : Z_n\to\left[0,1+\alpha\right[$ such that,
    \[\tau_{n,\alpha}(x)=\begin{cases}x-\left(1-\frac{1}{d^{n}}\right)+\frac{1}{d^{n+1}}\ &if\ x\in \left[1-\frac{1}{d^{n}},1-\frac{2}{d^{n+1}}\right[\\
    x+\frac{2}{d^{n+1}}+\alpha^{(n-1)}\ &if\ x\in\left[1-\frac{2}{d^{n+1}},1-\frac{1}{d^{n+1}}\right[\\
    x-\left(1+\alpha^{(n)}-\frac{1}{d^{n+1}}\right)+\frac{d-1}{d^{n+1}}\ &if\ x\in \left[1+\alpha^{(n)}-\frac{1}{d^{n+1}},1+\alpha^{(n)}\right[\\
    x+\frac{1}{d^{n+1}}\ &if\ \alpha_{n}=2\ and\ x\in \left[1+\alpha^{(n-1)},1+\alpha^{(n-1)}+\frac{1}{d^{n+1}}\right[.
    \end{cases}\]
    Let $x\in [0,1+\alpha[$.
    The $Z_n$ forms a partition. For $x\in [0,1+\alpha[$, we denote by the unique integer $N(x)$ such that $x\in Z_{N(x)}$.
    Thus, the Chacon transformation is given by \begin{equation}\label{index}T_\alpha(x)=\tau_{N(x),\alpha}(x).\end{equation}
\end{definition}

The height $h_k^\alpha$ of the Rockhlin tower of $k$ steps depends on $\alpha$ and $h_{k+1}^\alpha=dh_k^\alpha+\alpha_k$ with $h_0^\alpha=1$. For all $\alpha\in\{1,2\}^\mathbb{N}$, we denote by $\mathbb{L}eb_\alpha$ and $\lambda_\alpha$ the Lebesgue measure on $[0,1+\alpha[$,  normalized so that $(\alpha+1)\mathbb{L}eb_\alpha(A_k)=\frac{1}{d^k}$.

The main result of this paper is 
\begin{theorem}[Keplerian shear of Chacon transformation]
    \label{theorem_Chacon}
    The dynamical system $(\Omega,T,\mathcal{H}\otimes (\mathbb{L}eb_\alpha)_{\alpha\in\{1,2\}^\mathbb{N}})$ exhibits keplerian shear, for $d\geq 7$ odd.
\end{theorem}

The rest of the work is dedicated to the proof of Theorem \ref{theorem_Chacon}, using the weak mixing approach described in the introduction.

\section{Returns and induction}\label{sec_temps}

Following the idea in \cite{Park}, we consider $A_k=\left[0,\frac{1}{d^k}\right[$ for $k\in\mathbb{N}$. The quantities $$\lambda_\alpha(A_k\cap T_\alpha^{-n}(A_k))$$ are completely  understood via the induced map of $T_\alpha$ on $A_k$ and the sequences of successives returns times to $A_k$.

\begin{definition}[First return]
    We define the first return in $A_k$ of $x\in A_k$ the function \[r_{A_k,\alpha}(x)=\min\{n\in\mathbb{N}^* : T^n_\alpha(x)\in A_k\}.\]
\end{definition}

\begin{definition}[Induction] We define the induced transformation on $A_k$ by
    \[S_{A_k,\alpha}(x)=T_\alpha^{r_{A_k}(x)}(x).\]
\end{definition}

Any $x\in[0,1[$ has a symbolic representation: using the $d-$adic subdivision of $[0,1[$ with intervals $\left[\frac{i}{d^{j}},\frac{i+1}{d^{j}}\right[$ for $i\in \intervEnt{0}{d-1}$ and $j\in \mathbb{N}^*$,
we get the unique sequence $(x_j)_{j\in\mathbb{N}^*}\in \intervEnt{0}{d-1}^{\mathbb{N}^*}$ such that \[x=\sum_{j\in\mathbb{N}^*}x_jd^{-j}.\]
With a little abuse of language, we identify the  number $x$ with the sequence of its digits $(x_j)_{j\in\mathbb{N}^*}$. 
\begin{definition}
    We define the shift of $x\in[0,1[$ by \[\sigma x=\sum_{j\in\mathbb{N}^*}x_{j+1}d^{-j}.\]
    Similarly, we define the shift of $\alpha=(\alpha_j)_{j\in\mathbb{N}}$ by \[\sigma\alpha=(\alpha_{j+1})_{j\in\mathbb{N}}.\]
\end{definition}


We state a series of preparatory lemmas.

\begin{lemme}[Index property]\label{lemme_index}
    For $x\in [0,1[$, we have \[(x_1=d-1\implies N(\sigma x)=N(x)-1).\]
\end{lemme}

\begin{proof}
    Let $x\in [0,1[$.
    By (\ref{indice_val}), $x\in Z_{N(x)}\cap [0,1[=\left[1-\frac{1}{d^{N(x)}},1-\frac{1}{d^{N(x)+1}}\right[$.
    Thus \[d-\frac{1}{d^{N(x)-1}}-x_1\leq \sigma x <d-\frac{1}{d^{N(x)}}-x_1.\]
    We get
    \[
        1-\frac{1}{d^{N(x)-1}}\leq \sigma x<1-\frac{1}{d^{N(x)}}.\]
    In addition, 
    \[
        1-\frac{1}{d^{N(\sigma x)}}\leq \sigma x<1-\frac{1}{d^{N(\sigma x)+1}}.
    \]
    Therefore \[N(x)-1<N(\sigma x)+1\ \text{and}\ N(\sigma x)<N(x).\]
    In other words,\[N(\sigma x)=N(x)-1.\]
\end{proof}

\begin{lemme}
\label{formula_A_k}
    For $k\in\mathbb{N}$, for $x\in A_k$, \[T_\alpha^{h_k^\alpha-1}(x)=x+1-\frac{1}{d^k}.\]
\end{lemme}
\begin{proof}
    We prove it by induction on $k$.
    Taking $k=0$, $x\in A_0$, we get $h_0^\alpha=1$ and then, $T_\alpha^{h_0^\alpha-1}(x)=x=x+1-\frac{1}{d^0}$.
    Now, we suppose that the property is true for some $k\in\mathbb{N}$.
    Let $x\in A_{k+1}$.
    We have $A_{k+1}=\left[0,\frac{1}{d^{k+1}}\right[\subset \left[0,\frac{d-2}{d^{k+1}}\right[$.
    Then, $x\in \left[0,\frac{d-2}{d^{k+1}}\right[$.
    Thus, by recurrence hypothesis and Definition \ref{def_trans}, \begin{equation}\label{equa_trans}\forall j\in\intervEnt{0}{d-2},T_\alpha^{jh_k^\alpha}(x)=x+\frac{j}{d^{k+1}}.\end{equation}
    Then, $T_\alpha^{(d-2)h_k^\alpha}(x)=x+\frac{d-2}{d^{k+1}}\in \left[\frac{d-2}{d^{k+1}},\frac{d-1}{d^{k+1}}\right[\subset A_k$.
    Therefore, \[T_\alpha^{(d-1)h_k^\alpha-1}(x)=x+\frac{d-2}{d^{k+1}}+1-\frac{1}{d^k}=x+1-\frac{2}{d^{k+1}}\in\left[1-\frac{2}{d^{k+1}},1-\frac{1}{d^{k+1}}\right[.\]
    We get, by Definition \ref{def_trans}\begin{align}\label{equa_final_step}T_\alpha^{(d-1)h_k^\alpha+\alpha_k}(x)=x+\frac{d-1}{d^{k+1}}\in\left[\frac{d-1}{d^{k+1}},\frac{1}{d^k}\right[\subset A_k.\end{align}
    Then, \[T_\alpha^{h_{k+1}^\alpha-1}(x)=T_\alpha^{dh_k^\alpha+\alpha_k-1}(x)=x+1-\frac{1}{d^{k+1}}.\]
\end{proof}

We want to switch on a symbolic representation which is essential for the sequel. For this, we define the symbolic return.

\begin{definition}[The symbolic first return]
    We define the function \begin{equation} \label{f_return} r_{k,h}^\alpha(x)= \begin{cases}
    h &if\ x \in \left[0,1-\frac{2}{d}\right[\\
    h+\alpha_k &if\ \left[1-\frac{2}{d},1-\frac{1}{d}\right[\\
    r_{k,h}^{\sigma\alpha}(\sigma x)\ &else
\end{cases}.\end{equation}
    Then, the symbolic first return is $r_{k,h_k^\alpha}^\alpha$ more precisely, $r_{A_k,\alpha}\circ u_k^{-1}=r_{k,h_k^\alpha}^\alpha$.
\end{definition}
Here, we introduce the parameter $h$ because $h_k^\alpha$ depends on $\alpha$, and when $x\geq 1-\frac{1}{d},$  \[r_{k,h_k^\alpha}^\alpha(x)=r_{k,h_k^\alpha}^{\sigma\alpha}(\sigma x)\ne r_{k,h_k^{\sigma\alpha}}^{\sigma\alpha}(\sigma x).\]

\begin{lemme}
\label{ident_f_return}
    For $x\in [0,1[$,\[r_{k,h}^\alpha(x)=\begin{cases}
        h\ &if\ x\in \left[1-\frac{1}{d^{N(x)}},1-\frac{2}{d^{N(x)+1}}\right[\\
        h+\alpha_{k+N(x)}\ &if\ x\in \left[1-\frac{2}{d^{N(x)+1}},1-\frac{1}{d^{N(x)+1}}\right[
    \end{cases}.\]
\end{lemme}

\begin{proof}
    We prove by induction over $n$ that, for all $x$ such that $N(x)=n$, the formula \eqref{f_return} holds.
    We suppose that $n=0$.
    Let $x\in [0,1[$ such that $N(x)=n$.
    With Definition~\ref{f_return} on the symbolic first return, we get the result.
    Now, we take $n\in \mathbb{N}$ with the recurrence hypothesis.
    Let $x\in [0,1[$ such that $N(x)=n+1$.
    Then, $x\geq 1-\frac{1}{d^{n+1}}\geq 1-\frac{1}{d}$.
    Thus, $x_1=d-1$.
    We get that, $r_{k,h}^\alpha(x)=r_{k,h}^{\sigma\alpha}(\sigma x)$.
    Then, $N(\sigma x)=N(x)-1=n$.
    By induction hypothesis \[r_{k,h}^{\sigma\alpha}(\sigma x)=\begin{cases}
        h\ &if\ \sigma x\in \left[1-\frac{1}{d^{N(\sigma x)}},1-\frac{2}{d^{N(\sigma x)+1}}\right[\\
        h+(\sigma\alpha)_{k+N(\sigma x)}\ &if\ \sigma x\in \left[1-\frac{2}{d^{N(\sigma x)+1}},1-\frac{1}{d^{N(\sigma x)+1}}\right[
    \end{cases}.\]
    Then, \[r_{k,h}^{\alpha}(x)=\begin{cases}
        h\ &if\ x\in \left[1-\frac{1}{d^{N(x)}},1-\frac{2}{d^{N(x)+1}}\right[\\
        h+\alpha_{k+N(x)}\ &if\ x\in \left[1-\frac{2}{d^{N(x)+1}},1-\frac{1}{d^{N(x)+1}}\right[
    \end{cases}.\]
\end{proof}

The link between the first representation of the inducing time is  made by the map
\begin{equation} \label{homot} u_k : x\in A_k\mapsto d^{k}x\in [0,1[\end{equation}
as stated in the next lemma.

\begin{lemme}[Identification for first return]
    \label{id_f_return}
For $\alpha\in\{1,2\}^\mathbb{N}$, $k\in\mathbb{N}$ we get $r_{A_k,\alpha}=r_{k,h_k^\alpha}^\alpha\circ u_k$ on $A_k$.
\end{lemme}
\begin{proof}
    We prove it by induction on $k\in\mathbb{N}$.
    We begin with $k=0$.
    Let $x \in A_0=[0,1[$.
    First, we suppose that $x\in \left[0,1-\frac{2}{d}\right[$.
    Then, $T_\alpha(x)=x+\frac{1}{d}\in \left[\frac{1}{d},1-\frac{1}{d}\right[\subset A_0$.
    We get \[r_{A_k,\alpha}(x)=1=h_0^\alpha=r_{k,h_k^{\alpha}}^\alpha \circ u_0(x).\]
    Now, we suppose that $x\in \left[1-\frac{2}{d},1-\frac{1}{d}\right[$.
    Then, by (\ref{index}), \[T_\alpha(x)=x+\frac{2}{d}+\alpha^{(-1)}=x+\frac{2}{d}\in [1,1+\alpha^{(0)}[.\]
    For $q\in \intervEnt{0}{\alpha_0-1},T^{q+1}_\alpha(x)\in [1,1+\alpha^{(0)}[$.
    Then $r_{A_k,\alpha}(x)\geq \alpha_0+1$.
    In the other side, $T^{\alpha_0+1}_\alpha(x)=x+\frac{1}{d}\in\left[1-\frac{1}{d},1\right[\subset A_0$.
    Then \[r_{A_k,\alpha}(x)= \alpha_0+1=h_k^\alpha+1=r_{k,h_k^\alpha}^\alpha\circ u_0(x).\]
    Now, we suppose that $x\in \left[1-\frac{1}{d},1\right[$.
    Then, $N(x)\geq 1$.
    In the first case, $x\in \left[1-\frac{1}{d^{N(x)}},1-\frac{2}{d^{N(x)+1}}\right[.$
    Then, $T_\alpha(x)=x-\left(1-\frac{1}{d^{N(x)}}\right)+\frac{1}{d^{N(x)+1}}\in A_0$.
    For $x\in \left[1-\frac{2}{d^{N(x)+1}},1-\frac{1}{d^{N(x)+1}}\right[$, by Lemma \ref{lemme_index}, we get \[r_{A_k,\alpha}(x)=h_k^\alpha+\alpha_{N(x)}=h_k^{\alpha}+(\sigma\alpha)_{N(\sigma x)}=r_{k,h_k^\alpha}^\alpha\circ u_0(x).\]
    Now, we take $k\in \mathbb{N}$ verifying the property.
    Let $x\in A_{k+1}\subset A_k$.
    By Lemma \ref{formula_A_k} and recurrence hypothesis, we get that $r_{A_{k+1},\alpha}(x)\geq dh_k^\alpha+\alpha_k=h_{k+1}^\alpha.$ We get \[r_{A_{k+1},\alpha}(x)=\begin{cases}
        h_{k+1}^\alpha\ &if\ x\in \left[\frac{1}{d^{k+1}}\left(1-\frac{1}{d^{N(x)}}\right),\frac{1}{d^{k+1}}\left(1-\frac{2}{d^{N(x)+1}}\right)\right[ \\
        h_{k+1}^\alpha+\alpha_{k+1+N(x)}\ &if\ x\in \left[\frac{1}{d^{k+1}}\left(1-\frac{2}{d^{N(x)+1}}\right),\frac{1}{d^{k+1}}\left(1-\frac{1}{d^{N(x)+1}}\right)\right[
    \end{cases}.\]
    By Lemma \ref{ident_f_return}, we get \[r_{A_{k+1},\alpha}(x)=r_{k+1,h_{k+1}^\alpha}^\alpha(u_{k+1}(x))\]
\end{proof}

We make the same link for the induced transformation.
Let the odometer \begin{equation}\label{odo} S(x)=\begin{cases}
    (x_1+1)\sigma x &if\ x_1\in \intervEnt{0}{d-2}\\
    0S(\sigma x) &else
\end{cases}.\end{equation}

Now, we fix $k\in\mathbb{N}$.

\begin{proposition}[Induction/Odometer]
\label{reduc_odo}
    The induced map is conjugated with the odometer \[S=u_k\circ S_{A_k,\alpha}\circ \reciproque{u_k}.\]
\end{proposition}

\begin{proof}
Let $x\in A_k$.
By Lemma \ref{lemme_index} and Proposition \ref{id_f_return}, we deduce that, \[T_\alpha^{r_{A_k,\alpha}(x)}(x)=\begin{cases}
    x+\frac{1}{d} &if\ x \in \left[0,1-\frac{1}{d}\right[\\
    \frac{1}{d}T_\alpha^{r_{A_k,\sigma \alpha}(\sigma x)}(\sigma x) &else
\end{cases}.\]
We get \[S_{A_k,\alpha}(x)=T_\alpha^{r_{A_k,\alpha}(x)}(x)=u_k^{-1}\circ u_k\circ T_\alpha^{r_{A_k,\alpha}(u_k^{-1}(u_k(x))}(u_k^{-1}(u_k(x))).\]
Then, for $y\in [0,1[$,\[u_k\circ S_{A_k,\alpha}\circ \reciproque{u_k}(y)=u_k\circ T_\alpha^{r_{k,h_k^\alpha}^\alpha(y)}\circ u_k^{-1}(y)=S(y).\]

\end{proof}

Now, we can use exclusively this representation for $x\in [0,1[$
\begin{equation}\label{f_return_2}
 r_{k,h}^\alpha=h+\varepsilon_k^\alpha,\quad \text{ where } \varepsilon_k^\alpha(x)=\begin{cases}
        0\ if\ x_1\leq d-3\\
        \alpha_k\ if\ x_1=d-2\\
        \varepsilon_k^{\sigma\alpha}(\sigma x)\ else.
    \end{cases}
\end{equation}

Now, we have proved that the induction is just the odometer transformation.
Then, to continue our work, we recall properties about the odometer.
\begin{proposition}[Properties of the odometer]\label{prop:odo}
    For $x\in \Sigma_d$ and an integer $\ell$, we have \begin{enumerate}
    \item{$(S^d(x))_1=x_1$}
    \item{$\sigma S^d(x)=S(\sigma x)$}
    \item{$S^{d\ell}(x)=x_1 S^\ell(\sigma x)$}
    \end{enumerate}
\end{proposition}
\begin{proof}
    Let $x\in \Sigma_d$ and an integer $\ell$.
    Then \[S^d(x)=S^{x_1}(S^{d-x_1}(x))=S^{x_1}\left(0S(\sigma x)\right)=x_1S(\sigma x).\]
    We get \[(S^d(x))_1=x_1\ and\ \sigma S^d(x)=S(\sigma x).\]
    We'll do an induction on $\ell$. Let begin with $\ell=0$. Then \[S^{d\ell}(x)=x=x_1S^0(\sigma x).\]
    Now, we suppose that is true for $\ell$. Then \[S^{d(\ell+1)}(x)=S^{d\ell}(S^d(x))=(S^d(x))_1S^\ell(\sigma S^{d}(x))=x_1S^{\ell+1}(\sigma x).\]
\end{proof}
Now, we define the $\ell-$th returns and prove main properties about this.

\begin{definition}[$\ell-$th return]
    We define $\ell-th$ returns of $x\in A_k$ the sequence of integer $t_\ell^\alpha(x)$ such that \[t_{\ell+1}^\alpha(x)=\min\{m> t_\ell^\alpha(x) : T^m_\alpha(x)\in A_k\}, t_0^\alpha(x)=0.\]
\end{definition}

We obtain immediately for $\ell\in\mathbb{N}$ and $x\in A_k$ \begin{align}
t_\ell^\alpha(x)=\sum_{j=0}^{\ell-1}r_{A_k,\alpha}(S_{A_k,\alpha}^j(x)) ={t^{\alpha}_{\ell,h_{k}^\alpha}}^'(u_k(x)),
\end{align}
where for any $x\in\Sigma_d$ and $h\in\mathbb{N}$
\begin{align}
{t_{\ell,h}^{\alpha}}^{'}(x)=\sum_{j=0}^{\ell-1}r_{k,h}^\alpha(S^j(x)).
\end{align}
This immediately gives that for any $\ell_1,\ell_2\in\mathbb{N}$
    \begin{align}
    \label{sum_ret}
    {t_{\ell_1+\ell_2,h}^\alpha}^'(x)={t_{\ell_1,h}^\alpha}^'(x)+{t_{\ell_2,h}^\alpha}^'(S^{\ell_1}(x)).
    \end{align}

The inducing times are related each other by the following inductive process.

\begin{proposition}
\label{rec_return}
    We have for $q\in\intervEnt{0}{d-1},$ \[{t_{d\ell+q,h}^\alpha}^{'}(x)=\begin{cases}((d-1)\ell +q)h+\ell\alpha_k+{t_{\ell,h}^{\sigma\alpha}}^{'}(\sigma x)\ if\ q+x_1 \leq d-2\ or\ q=0\\
((d-1)\ell +q)h+(\ell+1)\alpha_k+{t_{\ell,h}^{\sigma\alpha}}^{'}(\sigma x)\ if\ q+x_1 = d-1\ and\ q\geq 1\\
((d-1)\ell +q-1)h+(\ell+1)\alpha_k+{t_{\ell+1,h}^{\sigma\alpha}}^{'}(\sigma x)\ if\ q+x_1 \geq d\ and\ x_1\leq d-2\\
((d-1)\ell +q-1)h+\ell\alpha_k+{t_{\ell+1,h}^{\sigma\alpha}}^{'}(\sigma x)\ if\ q+x_1 \geq d\ and\ x_1= d-1\end{cases}.\]
\end{proposition}
\begin{proof} Let $x\in [0,1[$.
We claim that \begin{equation}\label{mult_return}{t_{d\ell,h}^\alpha}^{'}(x)=(d-1)\ell h+\ell\alpha_k+{t_{\ell,h}^{\sigma\alpha}}^{'}(\sigma x).\end{equation}
    
    We have that, \[{t_{d\ell,h}^\alpha}^{'}(x)=\sum_{j=0}^{d\ell-1}r_{k,h}^\alpha(S^j(x))=\sum_{i=0}^{\ell-1}\sum_{p=0}^{d-1}r_{k,h}^\alpha(S^{di+p}(x)).\]
    Thus, by Proposition \ref{prop:odo} \[{t_{d\ell,h}^\alpha}^{'}(x)=\sum_{i=0}^{\ell-1}\sum_{p=0}^{d-1}r_{k,h}^\alpha(S^p(x_1S^{i}(\sigma x)))=(d-1)\ell h+\ell\alpha_k+{t_{\ell,h}^{\sigma\alpha}}^{'}(\sigma x),\]
proving the claim.
The conclusion follows from (\ref{sum_ret}) and (\ref{mult_return}).
\end{proof}

The balanced ternary expansion was used in \cite{Park}
 to identify the exceptional set related to the mixing times.
For any $\ell\in\mathbb{N}$, there exists a unique sequence $(a_{\ell,j})_{j\in\mathbb{N}}\in \intervEnt{-\frac{d-1}{2}}{\frac{d-1}{2}}^{\mathbb{N}}$ such that
\[\ell=\sum_{j\in\mathbb{N}}a_{\ell,j}d^j.\]
We call it the balanced d-adic expansion.
 For each fiber $\alpha$, it will help us to understand the behavior of $\ell-$th returns. That allows to use return times for $T_\alpha$ as a cocycle on the shift, opening a path to operator methods.

\begin{proposition}\label{form_return}
    We have \[{t_{\ell,h}^\alpha}^'(x)=\ell h+\sum_{j\in\mathbb{N}}\left(a_{\ell,j}\left(\sum_{i=0}^{j-1}\alpha_{k+i}d^{j-(i+1)}\right)+sgn(a_{\ell,j})\varepsilon_k^{\sigma^j\alpha}(S^{\frac{q_{\ell,j}}{d^j}}(\sigma^j x))\right)\]
    with $q_{l,j}=\sum_{i\geq j+1}a_{\ell,i}d^{i}+a_{\ell,j}d^j\mathds{1}_{\mathbb{R}^*_-}(a_{\ell,j})$.
\end{proposition}

\begin{proof}
The conclusion follows from (\ref{sum_ret}), (\ref{mult_return}) and the Proposition \ref{rec_return}.
    
\end{proof}
\section{Distribution of $\ell-$th returns}\label{sec_return}
Now, we note for $w$ a finite word $[w]$ the cylinder associated and we note $\lambda_k=\lambda_{A_k}:=\frac{{(\lambda_{\alpha})}_{\vert A_k}}{\lambda_\alpha(A_k)}=\frac{\lambda_{\vert A_k}}{\lambda(A_k)}$ the conditional measure of $\lambda_\alpha$ conditioned by $A_k$. We also see that this measure doesn't depend on $\alpha$.
And we use $\lambda:=(\lambda_k)_{u_k}=(u_k)_*\lambda_k$ the push-forward of $\lambda_k$ by $u_k$, this is the Lebesgue measure on $[0,1]$.
In this section, we find the section measures of a $\ell-$th returns.
The measure of $\ell-$th returns give a new form of correlations useful to control the convergence to zero.
Then, we can express the correlation with measures of sections.
Before, we define the cut expectation of reduced first return.
\begin{proposition}\label{expect_eps} We have for any $\kappa\in\mathbb{N}$
\[\mathbb{E}_\lambda((\varepsilon_k^\alpha)^\kappa)=\sum_{j\in\mathbb{N}}\alpha_{k+j}^\kappa d^{-(j+1)}=:\beta_k^{(\kappa)}(\alpha).\]
\end{proposition}
For simplicity we note $\beta_k(\alpha)=\beta_k^{(1)}(\alpha)$.
\begin{proof}
    By (\ref{f_return_2}), we get that for all $j\in\mathbb{N}$, for all $a\in \intervEnt{0}{d-3},$ and for all $x\in [0,1[,\varepsilon_k^\alpha((d-1)^jax)=r_{k,h}^\alpha(ax)-h=0$ and $\varepsilon_k^\alpha((d-1)^j(d-2)x)=\alpha_{k+j}$.
    Thus,
    \[\mathbb{E}((\varepsilon_k^\alpha)^\kappa)=\sum_{j\in\mathbb{N}}\intInd{[(d-1)^j(d-2)]}{(\varepsilon_k^\alpha(x))^{\kappa}}{x}=\sum_{j\in\mathbb{N}}\alpha_{k+j}^{\kappa}\lambda([(d-1)^j(d-2)])=\beta_k^{(\kappa)}(\alpha).\]
\end{proof}

    We denote the the probability that the $\ell-$th return is equal to $n$ by 
    \[d_\ell^\alpha(n)=\lambda_\alpha(t_\ell^\alpha=n)\ and\ {d_{\ell,h}^\alpha}^{'}(n)=\lambda\left({t_{\ell,h}^{\alpha}}^{'}=n\right).\]
We will work with  ${d_{\ell,h}^\alpha}^{'}(n)$ instead of $d_\ell^\alpha(n)$ for convenience.

It follows from Proposition \ref{rec_return} that the $d_\ell^\alpha$ satisfy some recursive relations.
\begin{proposition}
\label{rec_measure}
    For $q\in\intervEnt{0}{d-1}$,we have \begin{equation}
        {d_{d\ell+q,h}^\alpha}^{'}(n)=\begin{cases}
            {d_{\ell,h}^{\sigma\alpha}}^{'}(n-((d-1)\ell h+l\alpha_k))&if\ q=0\\
            \frac{1}{d}\left(\begin{array}{ll}(d-2){d_{\ell,h}^{\sigma\alpha}}^{'}(n-A_{(1,\ell)}^{(1,h)})\\+{d_{\ell,h}^{\sigma\alpha}}^{'}(n-A_{(1,\ell)}^{(2,h)})\\+{d_{\ell+1,h}^{\sigma\alpha}}^{'}(n-A_{(1,\ell)}^{(4,h)})\end{array}\right)&if\ q=1\\
            \frac{1}{d}\left(\begin{array}{ll}(d-1-q){d_{\ell,h}^{\sigma\alpha}}^{'}(n-A_{(q,\ell)}^{(1,h)})\\+{d_{\ell,h}^{\sigma\alpha}}^{'}(n-A_{(q,\ell)}^{(2,h)})\\+(q-1){d_{\ell+1,h}^{\sigma\alpha}}^{'}(n-A_{(q,\ell)}^{(3,h)})\\+{d_{\ell+1,h}^{\sigma\alpha}}^{'}(n-A_{(q,\ell)}^{(4,h)})\end{array}\right)&if\ 2\leq q\leq d-2\\
            \frac{1}{d}\left(\begin{array}{ll} {d_{\ell,h}^{\sigma\alpha}}^{'}(n-A_{(d-1,\ell)}^{(2,h)})\\+(d-2){d_{\ell+1,h}^{\sigma\alpha}}^{'}(n-A_{(d-1,\ell)}^{(3,h)})\\+{d_{\ell+1,h}^{\sigma\alpha}}^{'}(n-A_{(d-1,\ell)}^{(4,h)})\end{array}\right)&if\ q=d-1
        \end{cases}
    \end{equation}

    with $\begin{cases}
            A_{(q,\ell)}^{(1,h)}=(q+(d-1)\ell )h+\ell\alpha_k\\
            A_{(q,\ell)}^{(2,h)}=(q+(d-1)\ell )h+(\ell+1)\alpha_k\\
            A_{(q,\ell)}^{(3,h)}=(q-1+(d-1)\ell )h+(\ell+1)\alpha_k\\
            A_{(q,\ell)}^{(4,h)}=(q-1+(d-1)\ell )h+\ell\alpha_k
        \end{cases}$
\end{proposition}

\begin{definition}
    We define the set of possible $\ell-th$ returns at time $n\in\mathbb{N}$ by \begin{equation}
        P_n^\alpha=\{\ell\in\mathbb{N} : d_\ell^\alpha(n)>0\}.
    \end{equation}
\end{definition}

The probability of self return at time $n$ is then expressed by the possible $d_l^'$s.

\begin{proposition}[Formula for correlations]
    For $(k,n)\in \mathbb{N}^2,$ we have \begin{equation}
        \lambda_\alpha(A_k\cap T_\alpha^{-n}(A_k))=\sum_{\ell\in P_n^\alpha}d_\ell^\alpha(n).
    \end{equation}
\end{proposition}

\begin{proof}
    Let $(k,n)\in \mathbb{N}^2$ and let $x\in [0,1[$.
    We suppose that $n\geq 1$. Then,
    \[\mathds{1}_{A_k\cap T_\alpha^{(-n)}(A_k)}(x)=1\Longleftrightarrow \exists \ell\in\mathbb{N},n=t_\ell^\alpha(x).\]
    Therefore, \[\lambda_\alpha(A_k\cap T_\alpha^{-n}(A_k))=\lambda_\alpha\left(\union{\ell\in\mathbb{N}}{\{t_\ell^{\alpha}=n\}}\right)=\sum_{\ell\in\mathbb{N}}d_\ell^\alpha(n).\]
\end{proof}

We obtain then \begin{equation}
    \label{cor_trans}
    \lambda_\alpha(A_k\cap T_\alpha^{-n}(A_k))=\frac{1}{d^k(\alpha+1)}\sum_{\ell\in\mathbb{N}}{d_{\ell,h_k^\alpha}^{\alpha}}^'(n)=\lambda_\alpha(A_k)\sum_{\ell\in\mathbb{N}}{d_{\ell,h_k^\alpha}^{\alpha}}^'(n).
\end{equation}

\begin{definition}[Sum of non-zero term]
    Let $\ell=\sum_{j\in\mathbb{N}}a_{\ell,j}d^j$ with $a_{\ell,j}\in\intervEnt{-\frac{d-1}{2}}{\frac{d-1}{2}}$.
    Then, we define \[b_{\ell}=Card(\{j\in\mathbb{N}:a_{\ell,j}\ne 0\}).\]
\end{definition}


\begin{lemme}
\label{card_Pn}
    For $\alpha\in \{1,2\}^\mathbb{N}$, the set $P_n^\alpha$ is included in an interval $I_n^\alpha$ of length lower than $C_2m$ with $C_2>0$, where $m=\lfloor \log_d(n)\rfloor$.
\end{lemme}

\begin{proof}
    Let $\alpha\in\{1,2\}^\mathbb{N}$.
    We consider $\ell\in P_n^\alpha$, then ${d_{\ell,h_k^\alpha}^\alpha}^'(n)>0$ and $n\geq \ell$.
    Then, there exists $x\in [0,1[,$ such that \[n={t_{\ell,h_k^\alpha}^\alpha}^'(x)=\ell h_k^\alpha+\sum_{j=0}^m\left(a_{\ell,j}\sum_{i=0}^{j-1}\alpha_{k+i}d^{j-i}+sgn(a_{\ell,j})\varepsilon_k^{\sigma^j\alpha}(S^{\frac{q_{l,j}}{d^j}}(\sigma^ix))\right).\]
     We get with an induction on $k$, \[h_k^\alpha = \sum_{i=0}^{k-1}d^{k-1-i}\alpha_{i}+d^{k}.\]
    With the $d-$adic decomposition, we get 
    \[\ell h_k^\alpha=\sum_{j\in\mathbb{N}}a_{\ell,j}d^{j}h_k^\alpha=\sum_{j\in\mathbb{N}}\left(a_{\ell,j}\left(d^{j}\sum_{i=0}^{k-1}d^{k-j-1-i}\alpha_{i}+d^{k+j}\right)\right).\]
    Then, \[n=d^k\ell(\alpha+1)-\sum_{j=0}^ma_{\ell,j}\beta(\sigma^j\alpha)+\sum_{j=0}^m sgn(a_{\ell,j})\varepsilon_k^{\sigma^j\alpha}(S^{\frac{q_{l,j}}{d^j}}(\sigma^ix)).\]
    Since \[\beta(\sigma^j\alpha)\leq\frac{2}{1-\frac{1}{d}}\]
    we have \begin{equation}\label{encad_n_l}\ell d^k(\alpha+1)-(2+c)m\leq n\leq \ell d^k(\alpha+1)+(2+c)m,\end{equation}
    with $c=\frac{2}{1-\frac{1}{d}}.$
    Hence, \begin{equation}\label{encadre_l/n}\frac{n}{d^k(\alpha+1)}-\frac{c+2}{d^k(\alpha+1)}m\leq \ell\leq \frac{n}{d^k(\alpha+1)}+\frac{c+2}{d^k(\alpha+1)}m.\end{equation}
    We set \begin{equation}\label{I_n}I_n^\alpha=\frac{n}{d^k(\alpha+1)}+m\left[-\frac{c+2}{d^k(\alpha+1)},\frac{c+2}{d^k(\alpha+1)}\right].\end{equation}
\end{proof}

Now, we can define \begin{equation}\label{eq:return_new}
    B_\ell^\alpha :=\{n\in \mathbb{N} : \ell\in I_n^\alpha\}
\end{equation}

\section{Exceptional sets}\label{exception}
Now, we can define and study the exceptional set of $\alpha$ where the strong-mixing property is not respected.

\subsection{First exceptional set}
\label{sec:begin_except_set}
This first exceptional set handles times which are too close to the rigidity times created by the heights of the towers. We fix $C_0>0$ such that $\frac{(1+C_0)(d-1)^{C_0}}{C_0^{C_0}}<\frac{3}{2}$.
\begin{definition}[Exceptional sets]
    We define for $\alpha\in \{1,2\}^\mathbb{N}$ the exceptional set by \begin{equation}
        J_{k,\alpha}:=\union{N\in \mathbb{N}}{\left(\union{\left(\ell\in\intervEnt{d^{N}}{d^{N+1}-1}\ and\ b_\ell\leq C_0 N\right)}{B_\ell^\alpha}\right)}
    \end{equation}
    We define the set of $\alpha$ for which $n$ is exceptional as \begin{equation}\label{first_exceptional_set}
        W_{k,n}=\{\alpha\in \{1,2\}^\mathbb{N} : n\in J_{k,\alpha}\}.
    \end{equation}
\end{definition}

$W_{k,n}$ is the link between $\alpha$-fibers and the exceptional set $ J_{k,\alpha}$ which describe the "bad" set.
The purpose here is to prove that the measure of $W_{k,n}$ goes to zero when $n$ goes to $+\infty$.

Before, we recall that $\mathcal{H}$ is the countable product of the uniform measure on each $\{1,2\}$.

\begin{proposition} We have
\label{vanish_ex_set}
    \[\mathcal{H}(W_{k,n})\xrightarrow[n\to+\infty]{}0.\]
\end{proposition}

\begin{proof}
Let $n\in\mathbb{N}$ and set $m=\lfloor\log_d(n)\rfloor$.
For $\alpha\in \{1,2\}^\mathbb{N}$  and $\ell\in\mathbb{N}$, by definition of $I_n^\alpha$,
\[\ell\in I_n^\alpha\iff \alpha\in \frac{n}{\ell d^k}-1+\frac{m}{\ell}\left[-\frac{c+2}{d^k},\frac{c+2}{d^k}\right].\]




With the abuse of language \eqref{abus_alpha}, $\mathcal{H}$ is the Hausdorff measure on the Cantor $d-$adic with digits equals to $1,2$.
The $q-th$ generation of this Cantor $d-$adic, say $K_q$, is made of $2^q$ intervals of length $d^{-q}$. Taking $q$ the smallest integer such that $\frac{1}{\ell}\leq d^{-q}$, an interval of length $\frac{1}{\ell}$ intersect at most two intervals of $K_q$. Therefore, the  $\mathcal{H}-$measure of an interval of length $\frac1\ell$ is bounded by $2^{1-q}\le \frac{d}{\ell^{\log_d(2)}}$.

We will need the following combinatorial estimate:
The number of $\ell$ with $d$-adic expansion containing less than $cm$ non zero digits is bounded from above by 
\[
\sum_{b=0}^{c m}{m \choose b}(d-1)^b\le (d-1)^{cm} c^{-cm}\sum_{b=0}^{cm}{m \choose b}c^b\le \left(\frac{(d-1)(1+c)}{c^c}\right)^m\le(\frac32)^m
\]
with $c=C_0$.
Thus $\mathcal{H}(\frac{n}{\ell d^k}-1+\frac{m}{\ell}\left[-\frac{c+2}{d^k},\frac{c+2}{d^k}\right])\leq \frac{(c+2)m}{\ell^{\log_d(2)}d^{k}}.$
Then, we get \[\begin{array}{ll}\mathcal{H}(W_{k,n})&\leq \sum_{\ell\in\intervEnt{d^{\lfloor \log(n)\rfloor}}{d^{\lfloor \log(n)\rfloor+1}-1}\ and\ b_\ell\leq C_0\lfloor \log_d(n)\rfloor}\mathcal{H}\left(\frac{n}{\ell d^k}-1+\frac{m}{\ell}\left[-\frac{c+2}{d^k},\frac{c+2}{d^k}\right]\right)\\
&\leq \frac{(c+2)m}{n^{\log_d(2)}d^{k}}
{\left(
\frac32
\right)}
^{\log_d(n)}\\
&\leq \frac{(c+2)m}{2^{m}d^{k}}
{\left(
\frac32
\right)}
^{m}.\end{array}\]
\end{proof}


\subsection{Second exceptional set}
This second exceptional set will ensure that all the $t_l^\alpha$ have a very close stochastic behavior for all $\ell\in I_n^\alpha$ (See Subsection~\ref{ssvariance}). Let fixed $p_n,q_n\in \mathbb{N}$ such that $q_n>p_n$ and $p_n=\lfloor\log_d(cm)\rfloor+1$ and $q_n=2\lfloor{\frac{1+\log_d(2)}{\log(2)}\log_d(m)\rfloor}+1$ with $c>0$.
Let $\ell_0$ the integer defined by the next balanced $d$-adic expansion \[\ell_0=\nu\underbrace{(-\nu)(-\nu)\cdots(-\nu)}_{m-q_n-1\ terms}\underbrace{\nu\nu\cdots\nu\nu}_{q_n-p_n\ terms}\underbrace{(-\nu)(-\nu)\cdots(-\nu)}_{p_n\ terms}\]
with $\nu=\frac{d-1}{2}.$

Let for $n\in\mathbb{N}$, \[\mathcal{B}_n:=\left(\union{r\in\mathbb{Z}}{\left\{\frac{\ell_0}{n}+\frac{r}{n}d^{q_n}\right\}+\left[0,\frac{d^{p_n}}{n}\right]}\right)\cap\frac{\mathbb{N}}{n}.\]

Let \[\check{\mathcal{B}_n}:=\left(\union{r\in\mathbb{Z}}{\left\{\ell_0+rd^{q_n}\right\}+\left[0,d^{p_n}\right]}\right)\cap\mathbb{N}.\]

For $m\in\mathbb{N}$, let $K_m$ be the set of $\alpha\in [0,1]$ whose first $m-$digits in $d-$adic decomposition are $1$ or $2$.

We write ${t_\ell^\alpha}^'$ and ${d_\ell^\alpha}^'$ instead of ${t_{\ell,h_{k}^\alpha}^\alpha}^'$ and ${d_{\ell,h_{k}^\alpha}^\alpha}^'$.

For $\alpha\in K_m$, we have $\alpha\geq \sum_{j=0}^md^{-(j+1)}=\frac{1-d^{-m}}{d-1}\geq\frac{1}{d}$.
We define \[\varphi : x\in \left[\frac{1}{d},+\infty\right[\mapsto \frac{1}{d^k(x+1)}.\]

Since $\varphi$ is a homeomorphism from $\left[\frac{1}{d},+\infty\right[$ to $\left]0,\frac{1}{d^{k-1}(d+1)}\right]$,
$\varphi(K_m)$ is as $K_m$ a union of intervals of length $O(d^{-m})$.

\begin{proposition}
    \label{maj_Cantor}
    Let the integer $m=\lfloor\log_d(n)\rfloor$.
We have \[ \# \mathcal{B}_n\cap \left(\varphi(K_m)+\left[-C_3\frac{m}{n},C_3\frac{m}{n}\right]\right)\leq 2^{-{q_n}} n^{\frac{\log(2)}{\log(d)}}\]
    with $C_3=\frac{(c+2)(d-1)}{d^{k+1}}.$
\end{proposition}

\begin{proof}
    $d^q\xrightarrow[n\to+\infty]{}+\infty$ and $n\simeq d^m$.
Then, for large $n$,\[\frac{d^{q_n}}{n}>d^{-m}.\]
Let $m_0=\lfloor\log_d(n)\rfloor-q_n$.
Then, we get $m_0\simeq m$ because $q_n=o(m)$.
Therefore, \[\frac{d^{q_n}}{n}\simeq d^{-m_0}.\]
Then, $m\geq m_0.$
$K_{m_0}$ contains $2^{m_0-1}$ intervals, then $\left(\varphi(K_{m_0})+\left[-C_3\frac{m}{n},C_3\frac{m}{n}\right]\right)$ contains no more than $2^{m_0-1}$ intervals.
Let $I_m$  be an interval of $K_m$ and $\abso{I_m}$ its length.
The length of an interval of $\varphi(K_m)+\left[-C_3\frac{m}{n},C_3\frac{m}{n}\right]$ is less than $\abso{I_m}+C_3\frac{m}{n}$ since $\varphi$ is $1-$Lipschitz. 
We have $p_n<q_n$, then for $n$ large enough, the period of $\mathcal{B}_n$ defined by $\frac{r}{n}d^{q_n}$ with $r\in\mathbb{Z}$ is much larger than the length of the intervals contained in $\varphi(K_{m_0})+\left[-C_3\frac{m}{n},C_3\frac{m}{n}\right]$.
Hence $$\#\mathcal{B}_n\cap \left(\varphi(K_m)+\left[-C_3\frac{m}{n},C_3\frac{m}{n}\right]\right)\leq \# \mathcal{B}_n\cap \left(\varphi(K_{m_0})+\left[-C_3\frac{m}{n},C_3\frac{m}{n}\right]\right)\leq 2^{m_0-1}=2^{-q_n}n^{\frac{\log(2)}{\log(d)}}.$$
\end{proof}

Let $\check{J}_{k,\alpha}:=\union{\ell\in \check{\mathcal{B}_n}}{B_\ell^\alpha}$.
Let \begin{align}\label{second_exceptional_set}\check{W}_{k,n}=\{\alpha\in \{1,2\}^\mathbb{N} : n\in \check{J}_{k,\alpha}\}.\end{align}

\begin{proposition} We have
\label{second_vanish_ex_set}
     \[\mathcal{H}(\check{W}_{k,n})\xrightarrow[n\to+\infty]{}0.\]
\end{proposition}

\begin{proof}
    Let $n\in\mathbb{N}$, let $\alpha\in\{1,2\}^\mathbb{N}$.
    Let $\ell\in I_n^\alpha$.
    By definition, we get \[\beta(\alpha) = \frac{n}{\ell}-h_k^\alpha+O\left(\frac{m}{n}\right).\]
    For $n$ large enough, there exists a unique interval $J_\ell$ of $\varphi(K_m)+\left[-C_3\frac{m}{n},C_3\frac{m}{n}\right]$ such that $\frac{\ell}{n}\in J_\ell$.
    By Lemma \ref{card_Pn} and (\ref{encadre_l/n}), we also find that the length $\abso{J_\ell}$ of $J_\ell$ is $O(\frac{m}{n})$.
    Moreover $\varphi$ is a diffeomorphism. Then the length of an interval $I_\ell$ of $K_m$ is $O(\frac{m}{n})$.
    Recalling that $q_n=2\lfloor \frac{1+\log_d(2)}{\log(2)}\log_d(m)\rfloor+1$, we get \[\mathcal{H}(I_\ell)\leq C \abso{I_\ell}^{\log_d(2)}.\]
    By Proposition \ref{maj_Cantor}, \[\mathcal{H}\left(\check{W}_{k,n}\right)=\mathcal{H}\left(\union{\ell\in \check{\mathcal{B}}_n}{I_\ell}\right)\leq 2^{m-q_n}Cd^{p_n+1}\left(\frac{m}{n}\right)^{\log_d(2)}\leq Cdm^{-(1+\log_d(2))}.\]
    
\end{proof}


\section{Limit theorems for time dependent observables}\label{sec_llt}
Our approach relies on a version of the Local limit theorem in a correlated time dependent case.
We need a version very near the one of Hafouta and Kifer in \cite{Haf} and \cite{Kif} because we have dependent variables and not i.d.
We recall that $k$ is fixed. Writing $\ell=\sum_{j\in\mathbb{N}}a_{\ell,j}d^j$ with $a_{\ell,j}\in\intervEnt{-\frac{d-1}{2}}{\frac{d-1}{2}}$, we consider 
$U_i^\alpha(x)=sgn(a_{\ell,i})\varepsilon_k^{\sigma^i\alpha}(S^{\frac{q_{\ell,i}}{d^i}}(x))$ and $X_i^\alpha(x)=U_i^{\alpha}(\sigma^i x)$. 

In this section, we fix $\alpha$ and note $U_i, X_i$, $\beta^{(\kappa)}(\alpha)$ and $\varepsilon^\alpha$ instead of $U_i^\alpha-\mathbb{E}(U_i^\alpha)$, $X_i^\alpha-\mathbb{E}(X_i^\alpha)$, $\beta_k^{(\kappa)}(\alpha)$, $\varepsilon_k^\alpha$ and $Y_{\alpha}=\ell h_k^\alpha+\sum_{j\in\mathbb{N}}\left(a_{\ell,j}\left(\sum_{i=0}^{j-1}\alpha_{k+i}d^{j-(i+1)}\right)\right)$. Recall that $\mathbb{E}(X_i^\alpha)=sgn(a_{\ell,i})\beta_k(\sigma^i\alpha)$ by Proposition~\ref{expect_eps}.
By Proposition \ref{form_return}, we can see that \begin{equation}\label{form_return_simple}{t_\ell^\alpha}^'(x)=Y_\alpha+\sum_{j\in\mathbb{N}}X_j^\alpha.\end{equation}
We also note $S_{j,m}=\sum_{i=0}^{m-1}X_{i+j}$ and $S_m=S_{0,m}$.

Then, the main purpose of this section will be to prove the following local limit theorem.

\begin{proposition}
\label{llt}
    For $m,n\in\mathbb{N}$ such that $m= \lfloor\log_d(n)\rfloor$,$\alpha\not\in M_m\cup W_{k,n}$,$\ell\in I_n^\alpha$, we have \[\lambda(\{S_m=n-\mathbb{E}({t_\ell^\alpha}^')\})=\frac{1}{\sigma_{0,m}\sqrt{2\pi}}e^{-\frac{(\mathbb{E}({t_\ell^\alpha}^')-n)^2}{2\sigma_{0,m}^2}}+O\left(\frac{1}{m}\right)\] with $\sigma_{0,m}=\sqrt{\mathbb{V}(S_m)}$ and $M_m$ a measurable set that we will define in (\ref{M_m_define}). 
    In other words, \[{d_\ell^\alpha}^'(n)=\frac{1}{\sigma({t_\ell^\alpha}^')\sqrt{2\pi}}e^{-\frac{(\mathbb{E}({t_\ell^\alpha}^')-n)^2}{2\sigma({t_\ell^\alpha}^')^2}}+O\left(\frac{1}{\log_d(n)}\right)\]
    with $\sigma({t_\ell^\alpha}^')=\sqrt{\mathbb{V}({t_\ell^\alpha}^')}$.
\end{proposition}

We list loosely the properties that we need to prove the LLT with the good error estimate
\begin{enumerate}
    \item {$\mathcal{L}_{\alpha,it}^{(m)}(\mathds{1})\simeq \mathds{1}+it\mathbb{E}(S_m)-\frac{1}{2}t^2\mathbb{V}(S_m)$}
    \item{$\sup_{\abso{t}\in [\epsilon_0,\pi]}{\abso{\mathcal{L}_{\alpha,it}^{(m)}(\mathds{1})}}=O(\frac{1}{m})$}
    \item {$\sigma_{0,m}\geq c\sqrt{m}$.}
\end{enumerate}

\subsection{Variance of the non homogeneous Birkhoff sum}

More precisely, we need that the variance $\mathbb{V}(S_m)$ has a linear growth.
We should first control the covariance between every $X_i$ and $X_j$.
\begin{proposition}[Exponential vanishing of covariance]
\label{van_cor}
\[\abso{Cov(X_i,X_j)}\leq \frac{1}{d^{i-j}}\left((\beta^{(2)}(\sigma^i\alpha))\sigma(X_i)\right)\leq \frac{\norme{\alpha}_\infty^2}{d^{i-j}(d-1)}\]
\end{proposition}

\begin{proof}
    Let $i>j$ be two integers.
    We'll note $w^'$ the word such that $\abso{w^'}=i-j$ and $S^{\frac{q_{l,j}}{d^j}}(w^'x)\in [(d-1)^{i-j}]$ for all $x$.
    Then using the invariance of the measure by $\sigma^j$ and the transfer operator $\mathcal{L}^{i-j}$ we get
     \[\begin{array}{ll}
sgn(a_{\ell,i})sgn(a_{\ell,j})Cov(X_i,X_j)&=\intInd{[0,1[}{\varepsilon_{k,h_k^{\alpha}}^{\sigma^i\alpha}(S^{\frac{q_{\ell,i}}{d^i}}(\sigma^i x))\varepsilon_k^{\sigma^j\alpha}(S^{\frac{q_{\ell,j}}{d^j}}(\sigma^j x))}{x}-\beta(\sigma^i\alpha)\beta(\sigma^j\alpha) \\
&=\intInd{[0,1[}{\varepsilon_{k,h_k^{\alpha}}^{\sigma^i\alpha}(S^{\frac{q_{\ell,i}}{d^i}}(x))\mathcal{L}^{i-j}(\varepsilon_k^{\sigma^j\alpha}(S^{\frac{q_{\ell,j}}{d^j}}(\cdot))(x)}{x}-\beta(\sigma^i\alpha)\beta(\sigma^j\alpha).
\end{array}
    \]
When $w\ne w^'$ we have $\varepsilon_k^{\sigma^j\alpha}(S^{\frac{q_{l,j}}{d^j}}(wx))=\varepsilon_k^{\sigma^j\alpha}(S^{\frac{q_{l,j}}{d^j}}(w0))$
thus  
\[\intInd{[0,1[}{\varepsilon_k^{\sigma^i\alpha}(S^{\frac{q_{\ell,i}}{d^j}}(x))\varepsilon_k^{\sigma^j\alpha}(S^{\frac{q_{\ell,j}}{d^j}}(w x))}{x}=\varepsilon_k^{\sigma^j\alpha}(S^{\frac{q_{l,j}}{d^j}}(w0))\beta(\sigma^i\alpha)\]
and \[\frac{1}{d^{i-j}}\left(\sum_{\abso{w}=i-j\\ \ and\ w\ne w^'}\varepsilon_k^{\sigma^j\alpha}(S^{\frac{q_{\ell,j}}{d^j}}(w 0))+\intInd{[0,1[}{\varepsilon_k^{\sigma^j\alpha}(S^{\frac{q_{\ell,j}}{d^j}}(w^' x))}{x}\right)=\beta(\sigma^j\alpha).\]
    Therefore,
    \[sgn(a_{\ell,i})sgn(a_{\ell,j})Cov(X_i,X_j)=\frac{1}{d^{i-j}}\intInd{[0,1[}{\varepsilon^{\sigma^j\alpha}(S^{\frac{q_{l,j}}{d^j}}(w^'x))\left(\varepsilon^{\sigma^i\alpha}(S^{\frac{q_{l,i}}{d^i}}(x))-\beta(\sigma^i\alpha)\right)}{x}.\]

Then, by Hölder inequality,  \[\abso{Cov(X_i,X_j)}\leq \frac{1}{d^{i-j}}\left(\sqrt{(\beta^{(2)}(\sigma^i\alpha))}\sigma(X_i)\right)\leq \frac{\norme{\alpha}_\infty^2}{d^{i-j}(d-1)}.\]
\end{proof}

Now, we control the variance of each term $X_i$.

\begin{proposition}[Minimisation of variance]
    \label{min_var} \[\mathbb{V}(X_i)\geq \frac{d-2}{(d-1)^2}.\]
\end{proposition}
\begin{proof}
    Let \[f : u\in\ell^\infty(\mathbb{N})\mapsto \sum_{j\in\mathbb{N}}u_j^2d^{-(j+1)}-\left(\sum_{j\in\mathbb{N}}u_jd^{-(j+1)}\right)^2.\]
    Then, \[f(u)-f(1)=\sum_{j\in\mathbb{N}}(u_j-1)d^{-(j+1)}\left((u_j+1)-\sum_{i\in\mathbb{N}}(u_i+1)d^{-(i+1)}\right)\]
    For $u\in [1,2]^\mathbb{N},f(u)\geq f(1)=\frac{d-2}{(d-1)^2}$.
    By Proposition~\ref{expect_eps}
    we get \[\mathbb{V}(X_i)=f(\sigma^i\alpha)\geq \frac{d-2}{(d-1)^2}.\]
\end{proof}
Now to control the variance of $S_m$ we use these estimates on the correlation coefficients.
Using Propositions \ref{van_cor} and \ref{min_var}, we have

 \begin{align}\label{contr_cor}\abso{\frac{Cov(X_i,X_j)}{\sigma(X_i)\sigma(X_j)}}\leq\frac{\norme{\alpha}_\infty\sqrt{d-1}}{d^{i-j}\sqrt{d-2}} .\end{align}

We note now $p_{ij}:=\frac{Cov(X_i,X_j)}{\sigma(X_i)\sigma(X_j)}$.
With (\ref{contr_cor}), we have,

\begin{equation}\label{dominant_diag}\sum_{i\ne j}p_{ij}\leq \frac{2\norme{\alpha}_\infty}{\sqrt{d-1}\sqrt{d-2}}.\end{equation}

By our choice of $d\geq 7$, we have that $\frac{2\norme{\alpha}_\infty}{\sqrt{d-1}\sqrt{d-2}}<1$. This ensures that the covariance matrix is diagonally dominant.
Hence, by the property of diagonally dominant matrices, for $(a_i)_{i\in\mathbb{N}}\in\mathbb{R}^\mathbb{N}$ we have

    \begin{equation} \label{min_var_lemme}\mathbb{V}\left(\sum_{i\in\mathbb{N}}a_iX_i\right)\geq \left(1-\frac{2\norme{\alpha}_\infty}{\sqrt{d-1}\sqrt{d-2}}\right)\sum_{i\in\mathbb{N}}a_i^2\mathbb{V}(X_i)\end{equation}

Finally, by (\ref{min_var_lemme}) and Proposition \ref{min_var}, we have
\begin{proposition}[Minimisation of the variance]
    \label{minimize_var}
    For all $\ell\in \mathbb{N}$,
    \[\mathbb{V}\left({t_{\ell,h_k^\alpha}^{\alpha}}^{'}\right)\geq \left(1-\frac{2\norme{\alpha}_\infty}{\sqrt{d-1}\sqrt{d-2}}\right)\frac{d-2}{(d-1)^2}b_\ell.\]
\end{proposition}

Now, we have a minimization of the variance, the first asumption that we need for our version of Local limit theorem.

\subsection{Complex perturbations of Perron-Frobenius operator}

Then, we introduce the Perron-Frobenius operator to prove the Local limit theorem, in other words, to prove  Proposition~\ref{llt}. 
The functions $U_j$ are not continuous, hence the operator does not act on Holder continuous functions. Hopefully, we can study the operator on a space of bounded oscillation functions.

In this section, we note $\mathcal{D}=\intervEnt{0}{d-1}$  and $\Sigma_d=\mathcal{D}^{\mathbb{N}^*}$.
The oscillation of a function $g\colon\Sigma_d\to\mathbb{C}$ on a subset $B$ of $\Sigma_d$ is defined by
\[osc(g,B) := \sup{\{\abso{g(x)-g(y)} : (x,y)\in B^2\}}.\]
Fix $\delta\in (\frac1d,1-\frac{1}{d})$.
\begin{definition}[$\delta-$oscillations]
    The $\delta-$oscillation of a function $g : \Sigma_d\to\mathbb{C}$, is defined by
\[\abso{g}_\delta := \sup_{n\in\mathbb{N}}\left(\delta^{-n}\sum_{\abso{a}=n}osc(g,[a])d^{-n}\right).\]
\end{definition}
Note a function $g$ such that $\abso{g}_\delta<+\infty$ is bounded.
\begin{definition}[$\delta-$bounded-oscillation functions]
We denote by  $OSC_\delta$ the set of $\delta-$bounded-oscillation function $g : \Sigma_d\to\mathbb{C}$. 
Given $\gamma>0$, we endow $OSC_\delta$ with the balanced norm $\norme{g}_{\delta,\gamma} = \max(\|g\|_\infty,\gamma\abso{g}_\delta)$, which makes it a Banach space.
\end{definition}

Then, we will work in the Banach space $(OSC_\delta,\norme{g}_{\delta,\gamma})$.

For convenience, and to accomodate with previous  references, we extend the definition of $U_j$ for $j<0$ by $U_j=0$.
\begin{definition}[Perron-Frobenius operators]
\label{perfrob}
    We define the Perron-Frobenius operator for $z\in\mathbb{C}$, $\alpha\in\{1,2\}^\mathbb{N}$ and $j\in\mathbb{Z}$ by
    \[\mathcal{L}^{(j)}_{\alpha,z} : g\in OSC_\delta\mapsto \left(\mathcal{L}^{(j)}_{\alpha,z}(g) : x\in \intervEnt{0}{d-1}^{\mathbb{N}^*}\mapsto\frac{1}{d}\sum_{w\in \intervEnt{0}{d-1}}e^{z U_j^\alpha(wx)}g(wx)\right). \]
\end{definition}
Observe that $\mathcal{L}:=\mathcal{L}_{\alpha,0}$ does not depend on $\alpha$.
Note that $\mathcal{L}^{(j)}_{\alpha,z}=\mathcal{L}$ for $j<0$.

The Perron-Frobenius operators satisfy a kind of Doeblin-Fortet (Lasota-Yorke) inequality.
\begin{proposition}
    \label{Perron_Frob_osc}
    For $g\in OSC_\delta$, $j\in\mathbb{Z}$, for all $\alpha\in \{1,2\}^\mathbb{N}$ and $t\in\mathbb{R}$\begin{equation}\label{ineq_osc_Perron}\abso{\mathcal{L}_{\alpha,it}^{(j)}(g)}_\delta\leq \delta\abso{g}_\delta+\frac{2}{d}\norme{g}_\infty.\end{equation}
\end{proposition}

\begin{proof}
    Let $n\in\mathbb{N}^*$. There exists a unique word $a^'$ in the alphabet $\intervEnt{0}{d-1}$ such that $\abso{a^'}=n$, and one letter $\omega^'$ such that \[\forall x\in \Sigma_d,S^{\frac{q_{l,j}}{d^j}}(\omega^{'}a^'x)\in [(d-1)^{n+1}].\]
    For $\omega\ne\omega^{'}$ or $a\ne a^'$, we get \[\forall x,y\in \intervEnt{0}{d-1}^{\mathbb{N}^*},U_j(\omega a x)=U_j(\omega a y).\]
    Then, for $\omega\ne\omega^'$, \[\abso{e^{itU_j(\omega ax)}g(\omega ax)-e^{itU_j(\omega ay)}g(\omega ay)}=\abso{g(\omega a x)-g(\omega a y)}\leq osc(g,[\omega a]).\]
    For $\omega=\omega^'$, we get \[\abso{e^{itU_j(\omega ax)}g(\omega ax)-e^{itU_j(\omega ay)}g(\omega ay)}\leq\abso{g(\omega a x)}+\abso{g(\omega a y)}\leq 2\norme{g}_\infty.\]
    Then, \[\begin{array}{ll}\sum_{\abso{a}=n}osc\left(\mathcal{L}_{\alpha,it}^{(j)}(g),[a])\right)d^{-n}&\leq \frac{1}{d^{n+1}}\sum_{\abso{a}=n}\sum_{\abso{\omega}=1}osc\left(e^{itU_j(\cdot)}g(\cdot),[\omega a]\right)\\
    &\leq \frac{1}{d^{n+1}}\sum_{\abso{a}=n\ and\ (\omega a\ne \omega^'a^')}osc(g,[\omega a])+\frac{2}{d^{n+1}}\norme{g}_\infty\\
    &\leq \delta^{n+1}\abso{g}_\delta+\frac{2}{d^{n+1}}\norme{g}_\infty.\end{array}\]
    We have $\delta>\frac{1}{d}$.
Therefore
    \[\abso{\mathcal{L}_{\alpha,it}^{(j)}(g)}_\delta\leq \delta\abso{g}_\delta+\frac{2}{d}\norme{g}_\infty.\]
\end{proof}

We remark that \begin{equation}\label{contract_control}\norme{\mathcal{L}_{\alpha,it}^{(j)}(g)}_\infty\leq \norme{g}_\infty.\end{equation}

\begin{definition}[Composed Perron-Frobenius]
    We define the composed Perron-Frobenius operator for $z\in\mathbb{C}, j\in\mathbb{Z}$ and $n\in\mathbb{N}$ by \[\mathcal{L}^{j,n}_{\alpha,z}=\mathcal{L}^{(j+n-1)}_{\alpha,z}\circ \dots\circ \mathcal{L}^{(j)}_{\alpha,z}\]
\end{definition}

By induction of the Proposition \ref{Perron_Frob_osc}, we get uniform control on the composed Perron-Frobenius operator.

\begin{proposition}
    For $g\in OSC_\delta$, $t\in\mathbb{R}$ and for all $n\in\mathbb{N}$, for all $\alpha\in\{1,2\}^\mathbb{N}$
\begin{equation}\label{ineq_osc_induction}\abso{\mathcal{L}_{\alpha,it}^{j,n}(g)}_\delta\leq\delta^n\abso{g}_\delta+\beta\norme{g}_\infty,\end{equation}
where $\beta=\frac{2}{d(1-\delta)}.$
\end{proposition}

We prove the local limit theorem by the spectral method, that relies on the analyticity of the Perron-Frobenius operators.
Next lemma is helpful for that purpose.

\begin{lemme}
    \label{control_complex_term}
    There exists $C>0$ such that for all $j\in\mathbb{N}$ and for all $\alpha\in \{1,2\}^\mathbb{N}$,\[\forall g\in OSC_\delta,\norme{U_j^\alpha(\cdot)g}_{\delta,\gamma}\leq C\norme{g}_{\delta,\gamma}.\]
\end{lemme}

\begin{proof}
    By definition, $\norme{U_j^\alpha(\cdot)g}_{\delta,\gamma}=\max\left(\norme{U_j^\alpha(\cdot)g}_\infty,\gamma\abso{U_j^\alpha(\cdot)g}_\delta\right)$and $\abso{U_j^{\alpha}(\cdot)g}_\delta\leq \norme{g}_\infty\abso{U_j^\alpha}_\delta+\norme{U_j^\alpha}_\infty\abso{g}_\delta$.
    We get  \begin{equation}\label{norme_u_uniform}\norme{U_j^\alpha}_\infty\leq 2+\beta(\sigma^j\alpha)\leq 2+\frac{d}{d-1}.\end{equation}
    We pose $D=2+\frac{d}{d-1}$.
    We know that for $n\in\mathbb{N}^*$ and $x,y\in\Sigma_d$ that for $a$ a word such that $\abso{a}=n$, we get that $Osc(U^\alpha_j,[a])\leq 2D$ and if $a\ne S^{-\frac{q_{l,j}}{d^j}}((d-1)^n)$, $Osc(U_j^\alpha,[a])=0$.
    Then \[\delta^{-n}\sum_{\abso{a}=n}Osc(U_j^\alpha,[a])d^{-n}\leq \frac{2D}{(\delta d)^n}.\]
    For $n=0$, by \eqref{norme_u_uniform}, \[Osc(U_j^\alpha,\Sigma_d)\leq 2\norme{U_j^\alpha}_\infty\leq 2D.\]
    We know also that $\delta>\frac{1}{d}$, hence $\frac{1}{(\delta d)^{n-1}}<1$.
    Then, $\abso{U_j^\alpha}_\delta\leq 2D$.
    Writing $C=2\left(\frac{1}{\gamma}+D\right)>0$, we get \[\norme{U_j^\alpha(\cdot)g}_{\delta,\gamma}\leq 2\abso{g}_\delta+2D\norme{g}_\infty\leq C\norme{g}_{\delta,\gamma}.\]
\end{proof}

Thus, we obtain the lemma about the analyticity of $(z\in \mathbb{C}\mapsto \mathcal{L}_{z}^{(j)})$.

\begin{lemme}
    \label{Perron_analytic}
    For all $\alpha\in\{1,2\}^\mathbb{N}$, the function $z\in \mathbb{C}\mapsto \mathcal{L}_{\alpha,z}^{(j)}\in\mathscr{L}(OSC_\delta)$ is analytic.
\end{lemme}

\begin{proof}
    Let $z\in \mathbb{C}$.
    By Lemma \ref{control_complex_term}, we get that the multiplication by $U_j^\alpha$ is a bounded operator on $OSC_\delta$, with norm bounded by some constant $C$.
    In addition, $\sum_{k\in\mathbb{N}}\frac{\abso{z}^k\abso{\norme{U_j^\alpha(\cdot)}}_{\delta,\gamma}^k}{k!}\leq \sum_{k\in\mathbb{N}}\frac{C^k\abso{z}^k}{k!}=e^{C\abso{z}}<\infty$.
    By Proposition \ref{Perron_Frob_osc}, for $t=0$, $\mathcal{L}$ is a continuous operator on $OSC_\delta$ and  \[\mathcal{L}_{\alpha,z}^{(j)}(g)=\mathcal{L}\circ \left(\sum_{k\in\mathbb{N}}\frac{z^k}{k!}(U_j^\alpha(\cdot))^kg\right).\]
    Therefore, $(z\in \mathbb{C}\mapsto \mathcal{L}_{\alpha,z}^{(j)})$ is analytic.
\end{proof}



Since we are dealing with the composition of different operators and not a single one, the usual spectral perturbation approach does not apply. We need to use complex cone approach introduced in \cite{Rugh} and developped in \cite{Dub} to keep track of the spectral data. We follow the strategy in \cite{Kif}, adapting the cone technique to our space of bounded oscillation instead of Hölder space.

We define for $a>0$, 
\[\mathcal{C}_{a,\mathbb{R}}:=\{g\geq 0 : \abso{g}_\delta\leq a\intInd{\Sigma_d}{g}{\lambda}\}.\]
We define for $n\in\mathbb{N}$ and for two sequences of infinite words $X=(x_w)_{w\in\mathcal{D}^n}$ and $X^'=(x_w^')_{w\in\mathcal{D}^n}$, thus we define 
\[l_{(X,X^')}^{(n)} : g\in OSC_\delta\mapsto \sum_{\abso{w}=n}\frac{1}{(\delta d)^n}(g(w\cdot x_w)-g(w\cdot x_w^')).\]
Let \[m_{a,X,X^'}^{(n)} : g\in OSC_\delta\mapsto a\intInd{\Sigma_d}{g}{\lambda}-l_{(X,X^')}^{(n)}(g).\] We deduce immediately that \begin{equation}\label{convexe_cone}\mathcal{C}_{a,\mathbb{R}}=\{g\geq 0 : \forall n\in\mathbb{N},\forall X,X^'\in (\Sigma_d)^{\mathcal{D}^n},m_{a,X,X^'}^{(n)}(g)\geq 0\}.\end{equation}
\begin{proposition}
\label{control_convex_cone}
    For $\frac{1}{d}<\delta<1-\frac{2}{d}$ and $a>\frac{2}{d(1-\delta)-1}$, there exists $\iota\in ]0,1[$ such that \[\mathcal{L}\left(\mathcal{C}_{a,\mathbb{R}}\right)\subset \mathcal{C}_{\iota a, \mathbb{R}}.\]
\end{proposition}

\begin{proof}
Let $a>\frac{2}{d(1-\delta)-1}$ and $g\in\mathcal{C}_{a,\mathbb{R}}$.
Then, by Proposition \ref{Perron_Frob_osc},
    \[l_{X,X^'}^{(n)}(\mathcal{L}(g))\leq\abso{\mathcal{L}(g)}_\delta\leq \delta\abso{g}_\delta+\frac{2}{d}\norme{g}_\infty.\]
    We have for each $(x,x^')\in \Sigma_d^2$, \[g(x)=g(x)-g(x^')+g(x^')\leq \abso{g}_\delta+g(x^'),\]
    thus, \begin{align}\label{major_osc}\forall x^'\in\Sigma_d,\norme{g}_\infty\leq \abso{g}_\delta+g(x^').\end{align}
    Then, we get,
    \[\forall y\in \Sigma_d,l_{X,X^'}^{(n)}(\mathcal{L}(g))\leq \left(\delta+\frac{2}{d}\right)\abso{g}_\delta+\frac{2}{d}g(y).\]
    Integrating with respect to $y$, we obtain,
    \begin{align}\label{l_ineq}l_{X,X^'}^{(n)}(\mathcal{L}(g))\leq \left(\delta+\frac{2}{d}\right)\abso{g}_\delta+\frac{2}{d}\intInd{\Sigma_d}{g}{\lambda}\leq \left(\frac{2}{d}+a\left(\delta+\frac{2}{d}\right)\right)\intInd{\Sigma_d}{g}{\lambda}.\end{align}
    Moreover, $a>\frac{2}{d(1-\delta)-1}$ and $\frac{1}{d}<\delta<1-\frac{2}{d}$, hence \begin{align} \label{m_ineq} m_{a,X,X^'}^{(n)}(\mathcal{L}(g))=a\intInd{\Sigma_d}{g}{\lambda}-l_{X,X^'}^{(n)}(\mathcal{L}(g))\geq a\left(1-\left(\frac{2}{da}+\delta+\frac{2}{d}\right)\right)\intInd{\Sigma_d}{g}{\lambda}.\end{align}
    We pose $\iota=\left(\frac{2}{da}+\delta+\frac{2}{d}\right)\in ]0,1[$.
    By \eqref{m_ineq}, \[m_{a,X,X^'}(\mathcal{L}(g))\geq a(1-\iota)\intInd{\Sigma_d}{g}{\lambda}.\]
    Then, \[m_{\iota a,X,X^'}^{(n)}(\mathcal{L}(g))\geq 0.\]
    We deduce that \[\mathcal{L}\left(\mathcal{C}_{a,\mathbb{R}}\right)\subset \mathcal{C}_{\iota a, \mathbb{R}}.\]
\end{proof}


We get also $\mathcal{L}(\mathcal{C}_{a,\mathbb{R}}\setminus \{0\})\subset \mathcal{C}_{\iota a,\mathbb{R}}\setminus \{0\}$

\begin{lemme}\label{lem:hilbert}
    For any $g\in \mathcal{C}_{\iota a,\mathbb{R}}$ we have
    \[
    h_{\mathcal{C}_{a,\mathbb{R}}}(g,1)\le\ln\left(\frac{\max(\sup g,(1+\iota)\int g)}{\min(\inf g,(1-\iota)\int g)}\right),
    \]
    where $h_{\mathcal{C}_{a,\mathbb{R}}}$ is the Hilbert metric of the corresponding cone.
\end{lemme}
\begin{proof}
By definition the distance between $g$ and $1$ is the infimum of $\ln\frac{Q_2}{Q_1}$, where $Q_1>0$ and $Q_2>0$ satisfies $g-Q_1,Q_2-g\in \mathcal{C}_{a,\mathbb{R}}$, and is infinite if this cannot happen.

Obviously the condition $g-Q_1\ge0$ holds when $Q_1\le \inf g$. The second condition reads 
\[
|g-Q_1|_\delta \le a \int g-Q_1.
\]
Since $g\in\mathcal{C}_{\iota a,\mathbb{R}}$, 
\[
|g-Q_1|_\delta =|g|_\delta\le \iota a\int g \le a\int g - aQ_1,
\]
provided $Q_1\le (1-\iota)\int g$. Therefore $Q_1\le\min(\inf g, (1-\iota)\int g)$.

The bounds for $Q_2$ are obtained similarly.
\end{proof}

Now, we suppose also that $\delta\leq \frac{d-1}{d+2}$.
\begin{proposition}\label{pro:fini}
For $a\in(\frac{2}{d(1-\delta)-1},\delta^{-1})$, the diameter in $\mathcal{C}_{a,\mathbb{R}}$ of the image $\mathcal{L}(\mathcal{C}_{a,\mathbb{R}})$ is finite.
\end{proposition}
\begin{proof}
Let $b=\frac{1}{2}(1-\delta a)$.
Let $g\in \mathcal{C}_{a,\mathbb{R}}\setminus\{0\}$.
We first claim that there exists $v\in\mathcal{D}$ such that $\inf_{[v]}g\ge b\int g$. Otherwise, for any $v\in\mathcal{D}$
\[
d\int_{[v]}g\le b\int g + osc(g,[v])
\]
and summing up over $v$ yields to, by definition of $|g|_\delta$, 
\[
\int g \le b\int g+\delta |g|_\delta \le (b+\delta a)\int g<\int g,
\]
a contradiction.
Therefore, we have
\[
\mathcal{L}(g)\ge \frac1d \inf_{[v]}g\ge \frac bd \int g.
\]
On  the other hand,
\[
\sup \mathcal{L}g \le \int \mathcal{L}g + |\mathcal{L}g|_\delta \le (1+\iota a)\int g.
\]
By Lemmas \ref{control_convex_cone} and \ref{lem:hilbert} we end up with
\[
h_{\mathcal{C}_{a,\mathbb{R}}}(\mathcal{L}g,1)\le \ln\left(\frac{\max(1+\iota a,1+\iota)}{\min(\frac bd,1-\iota)}\right)<\infty.
\]
The conclusion follows by the triangle inequality.
\end{proof}

We define the complexification of cone $\mathcal{C}_{a,\mathbb{R}}$ by
\[\mathcal{C}_{a,\mathbb{C}}=\mathbb{C}^*\left(\mathcal{C}_{a,\mathbb{R}}+i\mathcal{C}_{a,\mathbb{R}}\right)=\mathbb{C}^*\{x+iy\in\mathbb{C} : x\pm y\in\mathcal{C}_{a,\mathbb{R}}\}.\]

\begin{proposition}
    \label{cone_control}
    Let $D=2+\frac{d}{d-1}$.
    There exists $C>0$ such that 
    for any $U\in OSC_\delta$ with  $\norme{U}_{\delta,\gamma}\leq D$,
    for all $n\in\mathbb{N}$, for all $X,X^'$, for all $a>\frac{2}{d(1-\delta)-1}$, all $z\in \mathbb{C}$, for all $f\in \mathcal{C}_{a,\mathbb{R}}$, \[\abso{m_{a,X,X^'}^{(n)}(\mathcal{L}_z(f))-m_{a,X,X^'}^{(n)}(\mathcal{L}(f))}\leq Ce^{D\abso{\Re(z)}}\abso{z}m_{a,X,X^'}^{(n)}(\mathcal{L}(f))\]
    with $\mathcal{L}_z(f)=\mathcal{L}(e^{zU}f)$.
\end{proposition}

\begin{proof}
    Let $z\in \mathbb{C}$ and let $f\in \mathcal{C}_{a,\mathbb{R}}$.
    Then
    \[\begin{array}{ll}
         \abso{m_{a,X,X^'}^{(n)}(\mathcal{L}_z(f))-m_{a,X,X^'}^{(n)}(\mathcal{L}(f))}&\leq a\abso{\intInd{\Sigma_d}{(e^{zU}-1)f}{\lambda}}+\abso{l_{X,X^'}^{(n)}(\mathcal{L}((e^{zU}-1)f))}\\
         &\leq 2aD\abso{z}e^{D\abso{\Re(z)}}\intInd{\Sigma_d}{f}{\lambda}+\abso{\mathcal{L}((e^{zU}-1)f)}_\delta.
    \end{array}\]
    By Proposition \ref{Perron_Frob_osc} with $t=0$, \[\abso{\mathcal{L}((e^{zU}-1)f)}_\delta\leq \delta\abso{(e^{zU}-1)f}_\delta+\frac{2}{d}\norme{(e^{zU}-1)f}_\infty.\]
    We have \[\delta\abso{(e^{zU}-1)f}_\delta\leq \delta \norme{f}_\infty\abso{e^{zU}-1}_\delta+2D\abso{z}e^{D\abso{\Re(z)}}\abso{f}_\delta.\]
    We get also \[\frac{2}{d}\norme{(e^{zU}-1)f}_\infty\leq \frac{4D\abso{z}e^{D\abso{\Re(z)}}}{d}\norme{f}_\infty.\]
    We have a similar result for the oscillation of $e^{zU}-1$,
    \[\abso{e^{zU}-1}_\delta\leq \abso{z}e^{D\abso{\Re(z)}}|U|_\delta
    \le \frac{D}{\gamma}\abso{z}e^{D\abso{\Re(z)}} .\]
    Integrating with respect to $x^'$ the inequality (\ref{major_osc}), we get
    \[\norme{f}_\infty\leq\abso{f}_\delta+\intInd{\Sigma_d}{f}{\lambda}.\]
    We know that $f\in\mathcal{C}_{a,\mathbb{R}}$, then \[\abso{f}_\delta\leq a\intInd{\Sigma_d}{f}{\lambda}.\]
    Then \[\abso{m_{a,X,X^'}^{(n)}(\mathcal{L}_z(f))-m_{a,X,X^'}^{(n)}(\mathcal{L}(f))}\leq 2D\abso{z}e^{D\abso{\Re(z)}}\intInd{\Sigma_d}{f}{\lambda}\left(2a+\left(\delta+\frac{2}{d}\right)(a+1)\right).\]
    We note $C=\frac{2D\left(2a+(a+1)\left(\delta+\frac{2}{d}\right)\right)}{a(1-\iota)}$.
    By \eqref{m_ineq}, \[\abso{m_{a,X,X^'}^{(n)}(\mathcal{L}_z(f))-m_{a,X,X^'}^{(n)}(\mathcal{L}(f))}\leq Ce^{D\abso{\Re(z)}}\abso{z}m_{a,X,X^'}^{(n)}(\mathcal{L}(f)).\]
\end{proof}

Now, we prove that the inclusion of the image by Perron-Frobenius operator of the cone $\mathcal{C}_{a,\mathbb{C}}$ holds true for $z$ in a neigbourhood of $0$.

\begin{proposition}\label{include_cone_z}
    There exist a neighbourhood $U\subset \mathbb{C}$ of $0$ and $C_3>0$ such that for all $j\in\mathbb{Z}$, $\alpha\in\{1,2\}^\mathbb{N}$ and $z\in U$,
    \[\mathcal{L}_{\alpha,z}^{(j)}(\mathcal{C}_{a,\mathbb{C}})\subset \mathcal{C}_{a,\mathbb{C}},\quad \sup_{f,g\in \mathcal{C}_{a,\mathbb{C}}}{h_{\mathcal{C}_{a,\mathbb{C}}}(\mathcal{L}_{\alpha,z}^{(j)}(f),\mathcal{L}_{\alpha,z}^{(j)}(g)})\leq C_3.\]
\end{proposition}

\begin{proof}
By Proposition \ref{pro:fini} the diameter $D_0$ of the image by $\mathcal{L}$ of the real cone $\mathcal{C}_{a,\mathbb{R}}$ is finite. 
By Proposition \ref{cone_control}, we can choose a neighborhood $U\subset \mathbb{C}$ of $0$ such that 
\[\abso{m_{a,X,X^'}^{(n)}(\mathcal{L}_{\alpha,z}(f))-m_{a,X,X^'}^{(n)}(\mathcal{L}(f))}\leq \varepsilon m_{a,X,X^'}^{(n)}(\mathcal{L}(f))\]
with $\varepsilon>0$ such that $2\varepsilon\left(1+\cosh\left(\frac{D_0}{2}\right)\right)<1$.
Then, we get the result applying theorem A.2.4 from \cite{Kif}.

\end{proof}

The next lemma will be used to prove the Lemma \ref{sup_analy}.

\begin{lemme}
    \label{lem:coerc_dual}
    For $a>\frac{2}{d(1-\delta)-1}$ and $M\geq\frac{\gamma a+1}{2\gamma a}$ we have \[\left\{g \in OSC_\delta : \norme{g-\mathds{1}}_{\delta,\gamma}\leq \frac{1}{M}\right\}\subset \mathcal{C}_{a,\mathbb{C}}\]
\end{lemme}

\begin{proof}
    Let $M\geq\frac{\gamma a+1}{2\gamma a}$ and let $f\in OSC_\delta$ such that $\norme{f-\mathds{1}}_{\delta,\gamma}<\frac{1}{M}$.
    Thus, $\norme{\Re{f}-\mathds{1}}_{\delta,\gamma}\leq \norme{f-\mathds{1}}_{\delta,\gamma}<\frac{1}{M}$ and $\norme{\Im{f}}_{\delta,\gamma}\leq \norme{f-\mathds{1}}_{\delta,\gamma}<\frac{1}{M}$ and $\abso{\Re(f)\pm\Im(f)}_\delta\leq \frac{2}{\gamma M}$.
    Therefore, $\intInd{\Sigma_d}{(\Re(f)\pm\Im(f))}{\lambda}\geq 1-\norme{\Re(f)-\mathds{1}\pm\Im(f)}_\infty\geq 1-\frac{2}{M}$.
    We have that $M\geq\frac{\gamma a+1}{2\gamma a}$, then
    \[\abso{\Re(f)\pm\Im(f)}_\delta\leq a\intInd{\Sigma_d}{(\Re(f)\pm\Im(f))}{\lambda}.\]
\end{proof}

The next lemma is used to prove \eqref{eq:max} in Lemma \ref{sup_analy}.

\begin{lemme}[Lemma 2.8.2 \cite{Kif}]
    \label{analytic_control}
    Let $(X,\norme{\cdot})$ be a complex Banach space and let $Q : U \to X$ be an analytic function defined on an open set $U \subset \mathbb{C}$ which contains a closed ball $\overline{B}(z_0, \delta)$ of radius $\delta > 0$ around some $z_0 \in  \mathbb{C}$. Suppose that there exists $r>0$ such that $\abso{Q(z)}\leq r$ for any $z\in B(z_0,\delta)$. Let $k\geq 0$ be an integer and denote by $Q_{k,z_0}$ the Taylor polynomial of $Q$ of order $k$ around $z_0$. Then for any $z \in B(z_0, \delta)$,
    \[\norme{Q(z)-Q_{k,z_0}(z)}\leq \frac{(k+2)r\abso{z-z_0}^{k+1}}{\delta^{k+1}}.\]
\end{lemme}

\begin{lemme}[After Theorem 4.2.1 in \cite{Kif}]
\label{sup_analy}
There exist a complex neighborhood $U$ of $0$ and a constant $C>0$  such that for all $\alpha\in \{1,2\}^\mathbb{N}$:

For any $j\in\mathbb{Z}$, there exists\footnote{All these quantities depend on $\alpha$, but we do not write the dependence explicitely to lighten the notation.} a unique number $\lambda^{(j)}(z)$, a unique bounded-oscillation function $h_z^{(j)}$ and a unique measure $\nu_z^{(j)}$ such that \begin{align}\label{eq_Perron}\mathcal{L}^{(j)}_{\alpha,z}(h_z^{(j)})=\lambda^{(j)}(z)h_z^{(j+1)},\ {\mathcal{L}^{(j)}}^*_{\alpha,z}(\nu_z^{(j+1)})=\lambda^{(j)}(z)\nu_z^{(j)}\ and\ \nu_z^{(j)}(h_z^{(j)})=\nu_z^{(j)}(\mathds{1})=1.\end{align}
We note $\lambda^{j,n}(z)=\prodend{i\in\intervEnt{0}{n-1}}{\lambda^{(j+i)}(z)}$.
The functions $\lambda_j$, $h_{(\cdot)}^{(j)}$ and $\nu^{(j)}_{\cdot}$ are analytic, $\lambda_j$,$h_{(\cdot)}^{(j)}$ is a non zero number, $h_{0}^{(j)}=1,\lambda^{(j)}(0)=1$ and $\nu_0^{(j)}=\mathbb{L}eb$ and \begin{equation}\label{eq:max}\max\left({\sup_{z\in U}\abso{\lambda^{(j)}(z)},\sup_{z\in U}\norme{h^{(j)}_z},\sup_{z\in U}\norme{\nu^{(j)}_z}}\right)\leq C.\end{equation} When $z=t\in\mathbb{R},$ we have that there exist $c_1>c_0>0$ and $R_1>0$ such that, for all integer $j$ and all $t$, \begin{equation}\label{cont_h}c_0\leq h_{t}^{(j)}\leq c_1\ and\ \lambda^{(j)}(t)>R_1.\end{equation}
\end{lemme}

To prove this lemma, we introduce the definition \[{\mathcal{C}_{a,\mathbb{C}}}^':=\{\mu\in {OSC_\delta}^' : \forall f\in\mathcal{C}_{a,\mathbb{C}}\setminus\{0\}, \abso{\mu(f)}>0\}.\]



\begin{proof}
    By Proposition \ref{include_cone_z}, we get the assumption 4.1.2 page 175 from \cite{Kif}.
    By Lemma \ref{lem:coerc_dual}, we get applying Lemma A.2.7 page 276 from \cite{Kif} to the constant function $\mathds{1}$ that for $M\geq \frac{\gamma a+1}{\gamma a}$, \[\forall \mu\in {\mathcal{C}_{a,\mathbb{C}}}^', \norme{\mu}\leq M\abso{\mu(\mathds{1})}.\]
    Now, considering $\kappa : \mu\in {OSC_\delta}^'\mapsto \mu(\mathds{1})$, we get the equation (4.1.5) of Assumption 4.1.1 page 174 from \cite{Kif} with $\norme{\kappa}=1$. By definition of $\mathcal{C}_{a,\mathbb{R}}$, $l : f\in OSC_\delta\mapsto \max\{a+1,\gamma a\}\intInd{\Sigma_d}{f(x)}{x}$ satisfies
    \[\forall f\in \mathcal{C}_{a,\mathbb{R}},\norme{f}_{\delta,\gamma}\leq \abso{l(f)}\leq \max\{a+1,\gamma a\}\norme{f}_{\delta,\gamma}.\]
    Thus, we can apply Lemma A.2.2 page 274 from \cite{Kif} to $l$,
    \[\forall f\in \mathcal{C}_{a,\mathbb{C}}, \norme{l}\norme{f}_{\delta,\gamma}\leq 2\sqrt{2}\max\{a+1,\gamma a\}\abso{l(f)}.\]
    Therefore, $l$ satisfies the equation (4.1.4) of Assumption 4.1.1 page 174 from \cite{Kif} with $K=2\sqrt{2}\max\{a+1,\gamma a\}$.
    Applying Theorem 4.2.1 page 175 from \cite{Kif}, we get $\lambda^{(j)}$, a unique bounded-oscillation function $h_\cdot^{(j)}$ and a unique measure $\nu_\cdot^{(j)}$ described in the lemma and we get \eqref{eq_Perron} and \eqref{cont_h}.
    By Lemma \ref{analytic_control}, with $(X,\norme{\cdot})=(\mathbb{C},\abso{\cdot})$, we get that there exists $C>0$ such that,
    \[\sup_{z\in U}\abso{\lambda^{(j)}(z)}\leq C.\]
    We have the same for $\sup_{z\in U}\norme{h^{(j)}_z}$ with $(X,\norme{\cdot})=(OSC_\delta,\norme{\cdot}_{\delta,\gamma})$ and for $\sup_{z\in U}\norme{\nu^{(j)}_z}$ with $(X,\norme{\cdot})=({OSC_\delta}^{'},\abso{\norme{\cdot}}^{'}_{\delta,\gamma}).$
    In other words, we get \eqref{eq:max}.
\end{proof}

\begin{lemme} For any integer $j,m$
\label{mean_eigen}
    \[{\lambda^{j,m}}^'(0)=\mathbb{E}(S_{j,m}).\]
\end{lemme}

Recall the notation $S_{j,n}=\sum_{i=0}^{n-1}X_{i+j}$.

\begin{proof}
Let $j,m\in\mathbb{N}$.
Then,\[\mathcal{L}_{z}^{j,m}(h^{(j)}_z)=\lambda^{j,m}(z)h^{(j+m)}_z.\]
    We get, \[\frac{d}{dz}\mathcal{L}_{z}^{j,m}(h^{(j)}_z)={\lambda^{j,m}}^'(z)h^{(j+m)}_z+\lambda^{j,m}(z){h^{(j+m)}_z}^'.\]
    Thus, \[\mathcal{L}_{z}^{j,m}(S_{j,m}h^{(j)}_z+{h^{(j)}_z}^')={\lambda^{j,m}}^'(z)h^{(j+m)}_z+\lambda^{j,m}(z){h^{(j+m)}_z}^'\]
    To $z=0$, \[\mathcal{L}_{0}^{j,m}(S_{j,m}h^{(j)}_0+{h^{(j)}_0}^')={\lambda^{j,m}}^'(0)h^{(j+m)}_0+{h^{(j+m)}_0}^'.\]
    Therefore, \[\mathbb{E}(\mathcal{L}_{0}^{j,m}(S_{j,m}h^{(j)}_0+{h^{(j)}_0}^'))={\lambda^{j,m}}^'(0)\mathbb{E}(h^{(j+m)}_0)+\mathbb{E}({h^{(j+m)}_0}^')={\lambda^{j,m}}^'(0)+\mathbb{E}({h^{(j+m)}_0}^').\]
    Then, \[\mathbb{E}(S_{j,m}h^{(j)}_0)+\mathbb{E}({h^{(j)}_0}^')={\lambda^{j,m}}^'(0)\mathbb{E}(h^{(j+m)}_0)+\mathbb{E}({h^{(m)}_0}^')={\lambda^{j,m}}^'(0)+\mathbb{E}({h^{(j+m)}_0}^').\]
    By (\ref{eq_Perron}), \[\label{}\forall z\in D(0,\epsilon_0),\nu_z^{(j+m)}(h_z^{(j+m)})=1.\]
    We have that \[\nu_z^{(j+m)}(\mathcal{L}_{z}^{j,m}(h^{(j)}_z))=\lambda^{j,m}(z)\nu_z^{(j+m)}(h^{(j+m)}_z)=\lambda^{j,m}(z).\]
    Thus, for all $z\in D(0,\epsilon_0),$
    $\frac{d}{dz}\nu_z^{(j+m)}(h_z^{(j+m)})=0$
    hence
    \[ {\nu_z^{(j+m)}}^'(h_z^{(j+m)})+\nu_z^{(j+m)}\left({h_z^{(j+m)}}^'\right)=0.\]
    We get, \[{\nu_0^{(j+m)}}^'(h_0^{(j+m)})+\nu_0\left({h_0^{(j+m)}}^'\right)=0.\]
    But, $\nu_0^{(j+m)}=\mathbb{L}eb$ and $h_0^{(j+m)}=\mathds{1}$.
    Moreover, by (\ref{eq_Perron}), \[\forall z\in D(0,\epsilon_0),{\nu_z^{(j+m)}}^'(\mathds{1})=0.\]
    Then, \[\nu_0^{(j+m)}\left({h_0^{(j+m)}}^'\right)={\nu_0^{(j+m)}}^'(\mathds{1})+\nu_0^{(j+m)}\left({h_0^{(j+m)}}^'\right)=0\]
    and, \begin{align}\label{null_mean_h}\mathbb{E}({h_0^{(j+m)}}^')=0.\end{align}
    Then, \[{\lambda^{j,m}}^'(0)=\mathbb{E}(S_{j,m}).\]

\end{proof}

Following \cite{Haf}, we relate in the next lemmas the variance of $S_{j,n}$ and quasi eigenvalues $\lambda^{(j)}(z)$.
We fix $\epsilon_0>0$ such that the disk $D(0,\epsilon_0)\subset U$,
 $\lambda^{(j)}(0)=1$ and for all $z\in D(0,\epsilon_0), \lambda^{(j)}(z)\ne 0$, then, there exists an analytic function $\Pi_j$ on $D(0,\epsilon_0)$ such that for all $z\in D(0,\epsilon_0),$ \begin{align}\label{analy_lambda}\lambda^{(j)}(z)=e^{\Pi_j(z)}\ and\ \Pi_j(0)=0.\end{align}
We note $\Pi_{j,m}(z):=\sum_{i=0}^{m-1}\Pi_{j+i}(z)$.
Then, \begin{align}\label{analy_lambda2}\lambda^{j,m}=e^{\Pi_{j,m}}.\end{align}

\begin{lemme}\label{lem:pi_control}
There exist $R>0$ and $\beta>0$ such that for all $m\in\mathbb{N}$,
    \[\sup_{z\in D(0,\beta)}{\abso{\Pi_{0,m}(z)}}\leq Rm.\]
\end{lemme}

\begin{proof}
    By \eqref{eq:max}, there exists $K>0$ such that for all $j\in\mathbb{N},$ for all $z\in D(0,\epsilon_0)$,\[\abso{\lambda^{(j)}(z)-1}\leq K\abso{z}.\]
    Then, there exists $R>0$ and $\beta>0$ such that \[\forall j\in\mathbb{N},\forall z\in D(0,\epsilon_0), \abso{\Pi_j(z)}\leq R\abso{z}.\]
    Hence, \[\forall z\in D(0,\epsilon_0), \abso{\Pi_{0,m}(z)}\leq Rm.\]
\end{proof}

\begin{lemme}\label{lem:cov_control}
    For all $f,g\in OSC_\delta$, for all $i\in\mathbb{N}$
    \[\abso{Cov(f\circ \sigma^i,g)}\leq \norme{f}_{\infty}\abso{g}_\delta\delta^i.\]
\end{lemme}

\begin{proof}
    Let $\bar{g}=\sum_{\abso{w}=i}\mathds{1}_{[w]}g_w$ with $g_w=\frac{1}{d^i}\intInd{[0,1[}{g(wx)}{x}$.
    We have that $Cov(f\circ \sigma^i,\bar{g})=0$.
    For each $w$ such that $\abso{w}=i$, pour $x\in \Sigma_d$ \[\abso{g(wx)-g_w}\leq osc(g,[w]).\]
    Then, \[\norme{g-\bar{g}}_{\mathbb{L}^1([0,1[]}\leq \sum_{\abso{w}=i}osc(g,[w])\leq \abso{g}_\delta \delta^i.\]
    By Hölder and the triangle inequalities, \[\abso{Cov(f\circ \sigma^i,g)}\leq \norme{f}_{\infty}\abso{g}_\delta\delta^i.\]
\end{proof}

The next lemma is used to control uniformly the Taylor expansion of $\Pi_{0,m}$ near $0$.

\begin{lemme}
\label{control_order_2}
    For $j,m\in\mathbb{N}$, \[{\lambda^{j,m}}^{''}(0)=\mathbb{E}(S_{j,m}^2)+\mathcal{R}_{j,m}\]
    with $\mathcal{R}_{j,m}=\mathbb{E}({h_0^{(j)}}^{''}-{h_0^{(j+m)}}^{''})+2Cov(S_{j,m},{h_0^{(j)}}^')=O(1)$.
\end{lemme}

\begin{proof}
    Let $j,m\in\mathbb{N}$.
    Proceeding as in the proof of Lemma~\ref{mean_eigen} we get for $z\in D(0,\epsilon_0)$
    \[{\lambda^{j,m}}^{''}(z)h_z^{(j+m)}+2{\lambda^{j,m}}^{'}(z){h_z^{(j+m)}}^'+\lambda_{j,m}(z){h_z^{(j+m)}}^{''}=\mathcal{L}_{\alpha,z}^{j,m}(S_{j,m}^2h_z^{(j)}+2S_{j,m}{h_z^{(j)}}^'+{h_z^{(j)}}^{''}).\]
    Thus, taking the expectation at $z=0$ gives
    \[{\lambda^{j,m}}^{''}(0)+2{\lambda^{j,m}}^{'}(0)\mathbb{E}({h_0^{(j+m)}}^')+\lambda^{j,m}(0)\mathbb{E}({h_0^{(j+m)}}^{''})=\mathbb{E}(S_{j,m}^2)+2\mathbb{E}(S_{j,m}{h_0^{(j)}}^')+\mathbb{E}({h_0^{(j)}}^{''}).\]
    By \eqref{null_mean_h}, \[{\lambda^{j,m}}^{''}(0)+\mathbb{E}({h_0^{(j+m)}}^{''})=\mathbb{E}(S_{j,m}^2)+2Cov(S_{j,m}{h_0^{(j)}}^')+\mathbb{E}({h_0^{(j)}}^{''}).\]
    Therefore, \[{\lambda^{j,m}}^{''}(0)=\mathbb{E}(S_{j,m}^2)+\mathcal{R}_{j,m}.\]

    On the other hand, applying Lemma \ref{sup_analy} and integral Cauchy formula
    $\norme{{h_{0}^{(j)}}^{'}}_{\delta,\gamma}\leq C$ and $\abso{\mathbb{E}({h_0^{(j)}}^{''}-{h_0^{(j+m)}}^{''})}\leq \norme{{h_0^{(j)}}^{''}}_\infty+\norme{{h_0^{(j+m)}}^{''}}_\infty\leq 2C$, where $C$ doesn't depend on $j$ and $m$.
    Finally, Lemma~\ref{lem:cov_control} gives
    $Cov(S_{j,m},{h_0^{(j)}}^')=O(1)$.
    Hence, \[\mathcal{R}_{j,m}=O(1).\]
\end{proof}

Now, we need to control at the second order the term $\Pi_{0,m}(z)$ with the mean and variance of $S_m$.

\begin{lemme}
\label{dev_lim_pi}
    There exists $c_0>0$ such that for all $m\in \mathbb{N}$,  \[\forall z\in D(0,\beta),\abso{\Pi_{0,m}(z)-z\Pi^'_{0,m}(0)-\frac{z^2}{2}\Pi^{''}_{0,m}(0)}\leq c_0m\abso{z}^3.\]
\end{lemme}

\begin{proof}
    By Lemma \ref{lem:pi_control}, for all $m\in\mathbb{N},$ \[\sup_{z\in D(0,\beta)}{\abso{\Pi_{0,m}(z)}}\leq Rm.\]
    Then, for all $m\in\mathbb{N}^*,$ \[\sup_{z\in D(0,\beta)}{\abso{\frac{1}{m}\Pi_{0,m}(z)}}\leq R.\]
    By Lemma \ref{analytic_control}, we get that there exists $c_0>0$ such that for all $m\in \mathbb{N}^*$,  \[\forall z\in D(0,\beta),\abso{\frac{1}{m}\Pi_{0,m}(z)-\frac{1}{m}z\Pi^'_{0,m}(0)-\frac{1}{m}\frac{z^2}{2}\Pi^{''}_{0,m}(0)}\leq c_0\abso{z}^3.\]
    Thus, for all $m\in \mathbb{N}$, \[\forall z\in D(0,\beta),\abso{\Pi_{0,m}(z)-z\Pi^'_{0,m}(0)-\frac{z^2}{2}\Pi^{''}_{0,m}(0)}\leq c_0m\abso{z}^3.\]
\end{proof}



\begin{lemme}
    \label{develop_pi}
    There exists $c_1>0$ such that for all $m\in \mathbb{N}$,  \[\forall z\in D(0,\beta),\abso{\Pi_{0,m}(z)-z\mathbb{E}(S_m)-\frac{z^2}{2}\mathbb{V}(S_m)}\leq c_1(\abso{z}^2+m\abso{z}^3).\]
\end{lemme}

\begin{proof}
    ${\lambda^{0,m}}^'(z)=\Pi^'_{0,m}(z)e^{\Pi_{0,m}(z)}$. and ${\lambda^{0,m}}^{''}(z)=\Pi^{''}_{0,m}(z)e^{\Pi_{0,m}(z)}+\left(\Pi^{'}(z)\right)^2e^{\Pi_{0,m}(z)}$
    Then, $\Pi^{'}_{0,m}(0)={\lambda^{0,m}}^'(0)$ and $\Pi^{''}_{0,m}(0)={\lambda^{0,m}}^{''}(0)-\left({\lambda^{0,m}}^'(0)\right)^2.$
    By Lemma \ref{mean_eigen}, $\Pi^{'}_{0,m}(0)=\mathbb{E}(S_m)$.
    
    And by Lemma \ref{control_order_2}, $\Pi^{''}_{0,m}(0)=\mathbb{V}(S_m)+O(1)$.
    Then, there exists $\check{C}>0$ such that for all $m\in\mathbb{N}$, $\abso{\Pi^{''}_{0,m}(0)-\mathbb{V}(S_m)}\leq \check{C}$.
    By Lemma \ref{dev_lim_pi}, \[\forall z\in D(0,\beta),\abso{\Pi_{0,m}(z)-z\mathbb{E}(S_m)-\frac{z^2}{2}\mathbb{V}(S_m)}\leq c_1(\abso{z}^2+m\abso{z}^3).\]
\end{proof}


Here, we need to have a control of the image of $\mathds{1}$ by $\mathcal{L}_{z}^{0,n}$ because this term will appear in the proof of  the local limit theorem.
We can control very well $h^{(n)}_z$ and the purpose here is to prove that the quantity $\mathcal{L}_{\alpha,z}^{0,n}(\mathds{1})$ is not so far from $h^{(n)}_z$.
With $(4.10)$ page 17 or $(2.6)$ page 5 form \cite{Haf}, we obtain the next lemma which is an immediate consquence of the cone contraction.
\begin{lemme}
\label{control_Perron_eigen}
   There exist constants $A_1$ and $\zeta\in (0,1)$ such that for any $z\in U$, and $n\geq k_1$, \[\norme{\frac{\mathcal{L}_{\alpha,z}^{0,n}(\mathds{1})}{\lambda^{0,n}(z)}-h^{(n)}_z}_{\delta,\gamma}\leq A_1\zeta^n\]
\end{lemme}

\subsection{Moderate deviations}

The result here will be used in the next section to take care of the case where $t_\ell$ is relatively far from its expectation. These are consequences of spectral perturbations.

\begin{proposition}
\label{control_eigenvalue}
    There exists $a>0$ such that for $z\in U$, \[\forall j\in\mathbb{N},\abso{\lambda^{(j)}(z)}\leq 1+a\abso{z}^2.\]
\end{proposition}
\begin{proof}
    By Lemma \ref{mean_eigen}, ${\lambda^{(j)}}^'(0)=\mathbb{E}(X_j)=0$ since     $X_j$ are centered.
    
    Applying Lemma \ref{sup_analy} and integral Cauchy formula, we have that there exists $a>0$ such that for all $z\in U$, \[\abso{\lambda^{(j)}(z)}\leq 1+a\abso{z}^2.\]
\end{proof}

\begin{lemme} There exists $C>0$ such that for $t\in\mathbb{R}$ such that $|t|<\epsilon_0$ with $\epsilon_0$ defined in \eqref{analy_lambda} and $m$ integer
\label{exp_mom_cont}
    \[\mathbb{E}(e^{tS_m})\leq C\lambda^{0,m}(t).\]
\end{lemme}

\begin{proof}
    For $t\in [-\epsilon_0,\epsilon_0]$, we have by (\ref{cont_h}), since $\mathcal{L}_t^{0,m}$ is a positive operator, \[\mathbb{E}(e^{tS_m})=\mathbb{E}(\mathcal{L}^{0,m}_t(\mathds{1}))\leq\mathbb{E}\left(\mathcal{L}^{0,m}_t\left(\frac{1}{c_0}h^{(0)}_t\right)\right)=\frac{1}{c_0}\lambda^{0,m}(t)\mathbb{E}(h^{(m)}_t)\leq  \frac{c_1}{c_0}\lambda^{0,m}(t).\]
\end{proof}

Then, by the Proposition \ref{control_eigenvalue}, we obtain whenever $|t|<\epsilon_0$ and $m\in\mathbb{N}$ \begin{equation}\label{cont_exp}
    \mathbb{E}(e^{tS_m})\leq 2C(1+at^2)^m.
\end{equation}


\begin{proposition}\label{control_chernov} For any $m$ such that $m>1/\epsilon_0^2$ and any $\alpha$
     \[\lambda\left(\abso{S_m}\geq m^{\frac{3}{4}}\right)\leq 2Ce^{-m^{\frac{1}{4}}}.\]
\end{proposition}

\begin{proof}
For $\beta>0$, applying Chernov inequality with (\ref{cont_exp}), we get \[\lambda (S_m\geq \beta m)\leq e^{-t\beta m}(1+at^2)^m.\]

Now, we have taking $\beta=m^{-\frac{1}{4}}$ and $t=\frac{1}{\sqrt{m}}$, \begin{equation}
    \label{ineq_cont}
    \lambda(S_m\geq m^{\frac{3}{4}})\leq Ce^{-m^{\frac{1}{4}}}.
\end{equation}
Similarly with $t=-\frac{1}{\sqrt{m}}$ we get
\begin{equation}
    \label{ineq_cont-}
    \lambda(S_m\leq -m^{\frac{3}{4}})\leq Ce^{-m^{\frac{1}{4}}}.
\end{equation}
\end{proof}

\subsection{Local limit theorem}


We conclude this section with the proof of Proposition \ref{llt}.
The following lemma shows that the composition of the Perron-Frobenius operators from time $j$ to $n+j+i$ is a strict contraction, each time that $U_{n+j}$ can take both values $0$ and $1$ on some cylinder of depth $i$. This idea of aperiodicity was already present in Dolgopyat's work \cite{Dol}.

Fix $\theta\in(\frac{d\delta}{d-1},1)$. 
\begin{lemme} Let $g\in OSC_\delta$, $j,t$ and $n$ such that $\epsilon_0\leq \abso{t}\leq \pi$. 

If $ a_{\ell,n+j}\neq0$ and there exists $i\in\mathbb{N}$ such that $\alpha_{k+n+j+i}=1$ we have
    \begin{equation}\label{ineq_Perron}\norme{\mathcal{L}_{\alpha,it}^{j,n+i}(g)}_\infty\leq \left(1-\frac{\chi_0}{d^{n+i}}\right)\norme{g}_\infty+\frac{\theta^{n-1}}{d^{n+i}}\abso{g}_\delta.\end{equation}
    with $\chi_0=2-\max\{|1+e^{it}| : \epsilon_0\le\abso{t}\le \pi\}>0$.
\end{lemme}

\begin{proof}
Let $g\in OSC_\delta$. Let $n\in\mathbb{N}^*$ and $t\in ]-\pi,\pi]$ such that $\epsilon_0\leq \abso{t}\leq \pi$.
We note $G_{n-1}$ the set of word $a$ such that $\abso{a}=n-1$ and for all $r\in \intervEnt{1}{n-1}, a_r+\frac{q_{l,j+r}}{d^r}\not\equiv d-1 [d]$.
We claim that there exists a word $a\in G_{n-1}$ such that 
    \[osc(g,[a])<\theta^{n-1}\abso{g}_\delta.\]
Otherwise,  \[\forall a\in G_{n-1},osc(g,[a])\geq \theta^{n-1}\abso{g}_\delta.\]
  which would imply, since $\# G_n=(d-1)^{n-1}$, that \[\abso{g}_\delta\geq \delta^{-n+1}\sum_{\abso{a}=n-1}d^{-n+1}osc(g,[a])\geq \delta^{-n+1}d^{-n+1}\#G_{n-1}\theta^{n-1}\abso{g}_\delta>\abso{g}_\delta,\]
    a contradiction. We then fix such $a\in G_{n-1}$.
    
    We choose $b,b^'$ and $c$ such that \[\abso{b}=\abso{b^'}=1\ and\ \abso{c}=i\] verifying
    \[b+\frac{q_{l,j+n}}{d^n}\equiv 0[d]\ and\ S^{\frac{q_{l,j+n}}{d^n}}(b^'c\cdot )\subset [(d-1)^i(d-2)].\]
    Then \[S_n(ab^'cx)=S_{n-1}(ab^'cx)+sgn(a_{\ell,j+n})\ and\ S_{n-1}(ab^'cx)=S_{n-1}(abcx)=S_n(abcx)\] 
    since \[\varepsilon_k^{\sigma^{j+n}\alpha}(S^{\frac{q_{l,j+n}}{d^n}}(b^'cx))=\alpha_{k+n+j+i}=1\ and\ \varepsilon_k^{\sigma^{j+n}\alpha}(S^{\frac{q_{l,j+n}}{d^n}}(bcx))=0.\]
    Hence \[\mathcal{L}_{\alpha,it}^{j,n}(g)(cx)=\frac{1}{d^n}\left(\sum_{\abso{\omega}=n\ and\  \omega\ne ab\ and\ \omega\ne ab^'}e^{itS_n(\omega cx)}g(\omega cx)\right)+\frac{1}{d^n}\left(e^{itS_n(ab^'cx)}g(ab^'cx)+e^{itS_n(abcx)}g(abcx)\right).\]
    We know that $\epsilon_0\leq \abso{t}\leq \pi$, then we have, since $osc(g,[a])\leq \abso{g}_\delta\theta^{n-1}$ 
    \[\abso{\frac{1}{d^n}\left(e^{itS_n(ab^'cx)}g(ab^'cx)+e^{itS_n(abcx)}g(abcx)\right)}\leq\frac{1}{d^n}\left((2-\chi_0)\abso{g(ab^'cx)}+\abso{g}_\delta\theta^{n-1}\right)\leq \frac{1}{d^n}\left((2-\chi_0)\norme{g}_\infty+\abso{g}_\delta\theta^{n-1}\right).\]
    We get \begin{align}\label{Perron_major}\abso{\mathcal{L}_{\alpha,it}^{j,n}(g)(cx)}\leq \left(1-\frac{\chi_0}{d^n}\right)\norme{g}_\infty+\frac{\abso{g}_\delta}{d^n}\theta^{n-1}.\end{align}
    Then,\[\begin{array}{ll}\abso{\mathcal{L}_{\alpha,it}^{j,n+i}(g)(x)}&=\abso{\frac{1}{d^i}\sum_{\abso{\omega}=i}e^{it S_{n+j,i}(\omega x)}\mathcal{L}_{\alpha,it}^{j,n}(g)(\omega x)}\\
    &=\abso{\frac{1}{d^i}e^{itS_{n+j,i}(cx)}\mathcal{L}_{\alpha,it}^{j,n}(g)(c x)+\frac{1}{d^i}\sum_{\abso{\omega}=i\ and\ \omega\ne c}e^{itS_{n+j,i}(\omega x)}\mathcal{L}_{\alpha,it}^{j,n}(g)(\omega x)}\\
    &\leq\frac{1}{d^i}(1-\frac{\chi_0}{d^n})\norme{g}_\infty+\frac{\theta^{n-1}}{d^{n+i}}\abso{g}_\delta+(1-\frac{1}{d^i})\norme{g}_\infty\ by\ (\ref{Perron_major})\ and\ (\ref{contract_control})\\
    &=\left(1-\frac{\chi_0}{d^{n+i}}\right)\norme{g}_\infty+\frac{\theta^{n-1}}{d^{n+i}}\abso{g}_\delta.\end{array}\]
\end{proof}

Our goal is now to ensure that these conditions for the contraction appear frequently.
Let $k_1$ be a positive integer such that $C_0>2^{1-k_1}$. We recall that $C_0$ is defined at the beginning of Subsection \ref{sec:begin_except_set}.
\begin{definition}
    Let \[\Gamma_m(\alpha)=\sum_{j=0}^{m-1}\mathds{1}_{[2^{k_1}]}(\sigma^{j}\alpha).\]
\end{definition}

With the definition, we obtain immediately,
\begin{equation}
    \label{gamma_mean}
    \mathbb{E}(\Gamma_m)=2^{-k_1}m
\end{equation}

\begin{definition}
    Let \begin{align}\label{M_m_define} M_m :=\{\alpha\in\{1,2\}^\mathbb{N} : \Gamma_m(\alpha)\geq 2\mathbb{E}(\Gamma_m)\}.\end{align}
\end{definition}

\begin{lemme}
\label{control_Mm}
    There exist $v>0$ and $C_1>0$ such that \[\mathbb{P}(M_m)\leq C_1e^{-vm}.\]
\end{lemme}

\begin{proof}
    We get the result by large deviation theorem.
\end{proof}

Recall $m=\lfloor\log_d(n)\rfloor$ and for $\alpha\not\in W_{k,n}$, we have $b_\ell\geq C_0m$ for all $\ell\in I_n^\alpha$.

\begin{lemme}
\label{control_alpha} Let $c_1=(C_0-2^{1-k_1})>0$.
     \[\forall \alpha\not\in M_m\cup W_{k,n},\forall \ell\in I_n^\alpha,\#\{j\in\intervEnt{0}{m} : a_{\ell,j}\ne 0\ and\ \exists i\in \intervEnt{0}{k_1},\alpha_{k+j+i}=1\}\geq c_1m.\]
\end{lemme}

\begin{proof}
The set of $j$'s has a cardinality at least $b_\ell-\Gamma_m(\alpha)\geq (C_0-2^{1-k_1})m$.
\end{proof}

Let $K_0$ a fixed integer such that \begin{equation}\label{K0_def}\theta^{K_0-1}(c+\beta)<\frac{\chi_0}{2}\ \text{and}\ \delta^{K_0}+\frac{\beta}{c}<\frac{1}{2}\end{equation} with $c=\frac{8}{d(1-\delta)}$. 

\begin{lemme}
\label{Perron_frob_ineq}
    We have that for $n\in\mathbb{N}$ such that $n\geq K_0$, there exists $\frac{1}{2}\leq\Delta_n<1$ such that for $j,i,q\in\mathbb{N}$, $t\in]-\pi,\pi]$ verifying $a_{\ell,n+j}\ne 0,\alpha_{k+n+j+i}=1$, $\epsilon_0\leq \abso{t}\leq \pi$ and $i\leq k_1$,\[\norme{\mathcal{L}_{\alpha,it}^{j,n+i+q}(g)}_{\delta,\gamma}\leq \Delta_n \norme{g}_{\delta,\gamma}\]
    with $\gamma=\frac{d(1-\delta)}{4}$.
\end{lemme}

\begin{proof}
    Let $c=\frac{8}{d(1-\delta)}$ and $t\in]-\pi,\pi]$ such that $\epsilon_0\leq \abso{t}\leq\pi$.
    First, we suppose that $\abso{g}_\delta\geq c \norme{g}_\infty$.
    By (\ref{ineq_osc_induction}),\[\abso{\mathcal{L}_{\alpha,it}^{j,n}(g)}_\delta\leq \delta^n\abso{g}_\delta+\beta\norme{g}_\infty.\]
    We know that $n\geq K_0$, therefore, using \eqref{contract_control} and \eqref{K0_def}, we get
    \[\begin{array}{ll}
    \norme{\mathcal{L}_{\alpha,it}^{j,n}(g)}_{\delta,\gamma}&\leq \max\{\norme{g}_\infty,\gamma(\delta^n+\frac{\beta}{c})\abso{g}_\delta\}\\
    &\leq\max\{\frac{1}{c},\gamma(\delta^n+\frac{\beta}{c})\}\abso{g}_\delta\\
    &\leq\max\{\frac{1}{c\gamma},\delta^n+\frac{\beta}{c}\}\norme{g}_{\delta,\gamma}\\
    &=\max\{\frac{1}{2},\delta^n+\frac{\beta}{c}\}\norme{g}_{\delta,\gamma}\\
    &\leq \frac{1}{2}\norme{g}_{\delta,\gamma}.
    \end{array}\]
    Then, \begin{align}\label{first_ineq}
        \norme{\mathcal{L}_{\alpha,it}^{j,n}(g)}_{\delta,\gamma}&\leq\frac{1}{2}\norme{g}_{\delta,\gamma}.
    \end{align}
    We have that $\delta<1$, then $\frac{3\delta+1}{4}<1$
    Posing $\Delta=\frac{1}{2}$, we get 
    \[\norme{\mathcal{L}_{\alpha,it}^{j,n}(g)}_{\delta,\gamma}\leq \Delta \norme{g}_{\delta,\gamma}.\]
    Now, we suppose that $\abso{g}_\delta\leq c \norme{g}_\infty$.
    Then, by \eqref{ineq_Perron}, \eqref{contract_control} and \eqref{ineq_osc_induction}, for $q\in\mathbb{N},$ $\epsilon_0\leq\abso{t}\leq\pi$ \[\norme{\mathcal{L}_{\alpha,it}^{j,n+i+q}(g)}_\infty\leq \left(1-\frac{\chi_0}{d^{n+i}}\right)\norme{\mathcal{L}_{\alpha,it}^{j,q}(g)}_\infty+\frac{\theta^{n-1}}{d^{n+i}}\abso{\mathcal{L}_{\alpha,it}^{j,q}(g)}_\delta\leq \left(1-\frac{\chi_0}{d^{n+i}}+\frac{\theta^{n-1}}{d^{n+i}}(\delta^qc+\beta)\right)\norme{g}_\infty.\]
    We know that $n\geq K_0$ and $i\leq k_1$, then \[\norme{\mathcal{L}_{\alpha,it}^{j,n+i+q}(g)}_\infty\leq\left(1-\frac{\chi_0}{2d^{n+k_1}}\right)\norme{g}_\infty.\]
    Therefore, \[\begin{array}{ll}
    \norme{\mathcal{L}_{\alpha,it}^{j,n+i}(g)}_{\delta,\gamma}\leq\max\{1-\frac{\chi_0}{2d^{n+k_1}},\gamma(c\delta^{K_0}+\beta)\}\norme{g}_{\delta,\gamma}\\
    \end{array}.\]
    We have that $\Delta_n=\max\{1-\frac{\chi_0}{2d^{n+k_1}},\gamma(c\delta^{K_0}+\beta)\}\in (\frac{1}{2},1)$.
    Then, \[\norme{\mathcal{L}_{\alpha,it}^{j,n+i+q}(g)}_{\delta,\gamma}\leq \Delta_n \norme{g}_{\delta,\gamma}.\]
\end{proof}

We note for $\alpha\not\in M_m$, $\ell\in P_n^\alpha$ the set $Y_\ell=\{j\in \intervEnt{0}{m} :  a_{\ell,j}\ne0,\exists i\in \intervEnt{0}{k_1},\alpha_{k+j+i}=1\}$.
Let $y_1<y_2<\cdots<y_p$ the successive elements of $Y_\ell$, where $p$ is the cardinality of $Y_\ell$.
We consider the sequence $m_s=y_{(2K_0+k_1)s}$ for $s\in \intervEnt{1}{\lfloor\frac{p}{2K_0+k_1}\rfloor}$ and a sequence $(i_s)$, $0\leq i_s\leq k_1$ such that $\alpha_{k+m_s+i_s}=1$.
We pose $N_s=m_s-K_0,N_0=0$ and $n_s=N_s-N_{s-1}$.

Now, fix $g\in OSC_\delta$ and define $G_s^{(t)}$ by \[G_s^{(t)}=\mathcal{L}^{0,N_s}_{it}(g).\] for $s = 1,\dots,\lfloor\frac{p}{2K_0+k_1}\rfloor$.

\begin{proposition}
\label{prop_contr}
    There exists $\frac{1}{2}<\Delta<1$ such that for all $s\in \intervEnt{1}{\lfloor\frac{p}{2K_0+k_1}\rfloor-1}$ and for all $t\in]-\pi,\pi]$ such that $\epsilon_0\leq \abso{t}\leq \pi$,
    \begin{align}\label{equa_L}\norme{G_{s}^{(t)}}_{\delta,\gamma}\leq \Delta^{s-1}\norme{g}_{\delta,\gamma}.\end{align}
\end{proposition}

\begin{proof}
We apply Lemma \ref{Perron_frob_ineq}, with $n=y_{K_0}\ge K_0$, $j=0$, $i<k_1$ such that $\alpha_{k+n+i}=1$ and $q=y_{2K_0+k_1}-K_0-y_{K_0}-i\ge0$, which gives
\[\norme{G_{1}^{(t)}}_{\delta,\gamma}\leq \Delta_n\norme{g}_{\delta,\gamma}\le\norme{g}_{\delta,\gamma} .\]
    Let $s=1,\dots,\lfloor\frac{p}{2K_0+k_1}\rfloor-1$.
It suffices now to prove that for some $\Delta<1$
    \[\norme{G_{s+1}^{(t)}}_{\delta,\gamma}\leq \Delta\norme{G_{s}^{(t)}}_{\delta,\gamma}\]
and the result will follow by induction. We have 
\[
G_{s+1}^{(t)}=
\mathcal{L}^{N_s,n_{s}}_{it}\left(\mathcal{L}^{0,N_{s}}_{it}(g)\right)=\mathcal{L}^{N_s,n_{s}}_{it}\left(G_{s}^{(t)}\right).
\]
Applying Lemma \ref{Perron_frob_ineq} with $j=N_s$, $n=K_0$, $i=i_s$ and $q=n_s-K_0-i_s$ we get
\[
\norme{G_{s+1}^{(t)}}_{\delta,\gamma}\leq \Delta_{K_0}\norme{G_{s}^{(t)}}_{\delta,\gamma}.\]
\end{proof}

Now, we can prove the main result of this section.

\begin{proof} Proof of Proposition \ref{llt}

    Let $m,n\in\mathbb{N}$ such that $m= \lfloor\log_d(n)\rfloor$, $\alpha\not\in W_{k,n}$ and $\ell\in I_n^\alpha$.
    We consider \[\lambda_\alpha(\{S_m=n-\mathbb{E}({t_\ell^\alpha}^')\})=\frac{1}{2\pi}\integrale{-\pi}{\pi}{e^{-it(n-\mathbb{E}({t_\ell^\alpha}^'))}\mathbb{E}(e^{itS_m})}{t}\]
    by Fubini theorem. We get \[\mathbb{E}(e^{itS_m})=\mathbb{E}\left(\mathcal{L}_{\alpha,it}^{0,m}(\mathds{1})\right).\]
    By Fubini theorem, \[\lambda(\{S_m=n-\mathbb{E}({t_\ell^\alpha}^')\})=\frac{1}{2\pi}\mathbb{E}\left(\intInd{[-\pi,\pi]}{e^{-it(n-\mathbb{E}({t_\ell^\alpha}^'))}\mathcal{L}_{\alpha,it}^{0,m}(\mathds{1})}{t}\right).\]
    Since $\alpha\not\in M_m$, we have $p\geq c_2m$ with $c_2>0$.
    Then, by Lemma \ref{control_alpha} and (\ref{equa_L}), for $\epsilon_0\leq \abso{t}\leq \pi$, \[\abso{\mathbb{E}\left(\mathcal{L}_{\alpha,it}^{0,m}(\mathds{1})\right)}\leq C_2\Delta^{c_1m}\ \text{with}\ c_1>0,\ C_2>0,\ 0<\Delta<1.\]
    This implies, \[\mathbb{E}\left(\intInd{[-\pi,\pi]}{e^{-it(n-\mathbb{E}({t_\ell^\alpha}^'))}\mathcal{L}_{\alpha,it}^{0,m}(\mathds{1})}{t}\right)=\mathbb{E}\left(\intInd{[-\epsilon_0,\epsilon_0]}{e^{-it(n-\mathbb{E}({t_\ell^\alpha}^'))}\mathcal{L}_{\alpha,it}^{0,m}(\mathds{1})}{t}\right)+O(\Delta^{c_1m}).\]
    
    Now, we take $-\epsilon_0<t<\epsilon_0$.
    By Lemma \ref{control_Perron_eigen}, for $m\in\mathbb{N}$ large enough, we get that \[\mathcal{L}^{0,m}_{it}(\mathds{1})=\lambda^{0,m}(it)(h_{m}(it)+O(\zeta^m))\text{ with } 0<\zeta<1.\]
    By (\ref{analy_lambda2}), $\lambda^{0,m}(it)=e^{\Pi_{0,m}(it)}$ and by Lemma \ref{sup_analy}, $\Pi_{0,m}$ is analytic on $D(0,\epsilon_0)$.
    We know by Lemma \ref{develop_pi} that for all $z\in D(0,\epsilon_0),$ \[\Pi_{0,m}(z)=z\mathbb{E}(S_m)+\frac{z^2}{2}\mathbb{V}(S_m)+O(\abso{z}^2+m\abso{z}^3)=\frac{z^2}{2}\mathbb{V}(S_m)+O(\abso{z}^2+m\abso{z}^3).\]
    By Lemma \ref{sup_analy}, for all $z\in D(0,\epsilon_0),$\[h_m(z)=1+O(\abso{z}).\]
    Then, we get \[\mathcal{L}_{\alpha,it}^{0,m}(\mathds{1})=e^{-\frac{t^2}{2}\mathbb{V}(S_m)}(1+O(\abso{t}^2+m\abso{t}^3))(1+O(\abso{t})+O(\zeta^m)).\]
    Let $\psi_m : t\in \mathbb{R}\mapsto e^{-it(n-\mathbb{E}({t_\ell^\alpha}^'))}e^{-\frac{t^2}{2}\mathbb{V}(S_m)}$.
    Thus, \[\integrale{-\epsilon_0}{\epsilon_0}{\abso{t\psi_m(t)}}{t}\leq \intInd{[-\epsilon_0,\epsilon_0]}{\abso{t}e^{-\frac{t^2}{2}\mathbb{V}(S_m)}}{t}=\frac{1}{\sigma_{0,m}^2}\intInd{[-\sigma_{0,m}\epsilon_0,\sigma_{0,m}\epsilon_0]}{\abso{t}e^{-\frac{t^2}{2}}}{t}\leq\frac{2\sqrt{2\pi}}{\sigma_{0,m}^2}.\]
    And \[\intInd{[-\epsilon_0,\epsilon_0]}{m\abso{t^3\psi_m(t)}}{t}\leq\frac{m}{\sigma_{0,m}^4}\mathbb{E}(\abso{\mathcal{N}}^3)\ \text{with}\ \mathcal{N}\text{ the normal distribution.}\]
    Recall that $\alpha\not\in W_n$. Therefore, by Proposition \ref{minimize_var}, $\sigma_{0,m}\geq C_2\sqrt{m}.$
    Hence, \[\intInd{[-\epsilon_0,\epsilon_0]}{m\abso{t^3\psi_m(t)}}{t}=O\left(\frac{1}{m}\right).\]
    With the same proof, we get, \[\intInd{[-\epsilon_0,\epsilon_0]}{m\abso{t^4\psi_m(t)}}{t}=O\left(\frac{1}{m}\right)\]
    and \[\intInd{[-\epsilon_0,\epsilon_0]}{\abso{t^2\psi_m(t)}}{t}=O\left(\frac{1}{m}\right).\]
    Then, \[\intInd{[-\epsilon_0,\epsilon_0]}{e^{-it(n-\mathbb{E}({t_\ell^\alpha}^'))}\mathcal{L}_{\alpha,it}^{0,m}(\mathds{1})}{t}=\intInd{[-\epsilon_0,\epsilon_0]}{e^{it(\mathbb{E}({t_\ell^\alpha}^')-n)}e^{-\frac{t^2}{2}\mathbb{V}(S_m)}}{t}+O\left(\frac{1}{m}\right).\]
    Moreover,\[\intInd{[-\epsilon_0,\epsilon_0]}{e^{it(\mathbb{E}({t_\ell^\alpha}^')-n)}e^{-\frac{t^2}{2}\mathbb{V}(S_m)}}{t}=\frac{1}{\sigma_{0,m}}\intInd{[-\epsilon_0\sigma_{0,m},\epsilon_0\sigma_{0,m}]}{e^{it\frac{(\mathbb{E}({t_\ell^\alpha}^')-n)}{\sigma_{0,m}}}e^{-\frac{t^2}{2}}}{t}.\]
    By the property of the normal distribution,
    \[\frac{1}{\sigma_{0,m}}\intInd{[-\epsilon_0\sigma_{0,m},\epsilon_0\sigma_{0,m}]}{e^{it\frac{(\mathbb{E}({t_\ell^\alpha}^')-n)}{\sigma_{0,m}}}e^{-\frac{t^2}{2}}}{t}=\frac{\sqrt{2\pi}}{\sigma_{0,m}}e^{-\frac{(\mathbb{E}({t_\ell^\alpha}^')-n)^2}{2\sigma_{0,m}^2}}+O\left(\frac{e^{-\varepsilon_0^2\sigma_{0,m}^2}}{2\sigma_{0,m}}\right).\]
    Finally,\[\intInd{[-\epsilon_0,\epsilon_0]}{e^{-itn}\mathcal{L}_{\alpha,it}^{0,m}(\mathds{1})}{t}=\frac{\sqrt{2\pi}}{\sigma_{0,m}}e^{-\frac{(\mathbb{E}({t_\ell^\alpha}^')-n)^2}{2\sigma_{0,m}^2}}+O\left(\frac{1}{m}\right).\]
    Then, \[\lambda(\{S_m=n-\mathbb{E}({t_\ell^\alpha}^')\})=\frac{1}{\sigma_{0,m}\sqrt{2\pi}}e^{-\frac{(\mathbb{E}({t_\ell^\alpha}^')-n)^2}{2\sigma_{0,m}^2}}+O\left(\frac{1}{m}\right).\]
\end{proof}
\section{Keplerian shear for Chacon transformation}\label{sec_keplerian}
We do a compilation of previous results to prove the main theorem, the keplerian shear of the Chacon transformation.

First, we show that for times outside the exceptional sets, the self correlation of $A_k$ behave as in a strong mixing system.
We write $m=\lfloor \log_d(n)\rfloor$.
\begin{proposition}
\label{first_keplerian_shear}
For $\alpha\not\in W_{k,n}\cup M_m\cup \check{W}_{k,n}$,    \[\sum_{\ell\in\mathbb{N}}{d_\ell^\alpha}^'(n)=\lambda_\alpha(A_k)+O\left(\frac{\sqrt{\log_d(\log_d(n))}}{\left(\log_d(n)\right)^{\frac{1}{4}}}\right).\]
\end{proposition}
We prove it in Subsection~\ref{self_cor} since we first need to control the variances.

\subsection{Equivalence of variances}\label{ssvariance}

The purpose of this subsection is to prove that the variance of $t_\ell^\alpha$ are close as $\ell$ runs through $I_n^\alpha$ for good $\alpha$'s.

\begin{proposition}\label{limit_d_adic}
    For $\ell>\ell^'\in I_n^\alpha$, writing $\ell-\ell^'=\sum_{j=0}^mu_jd^j$ in balanced $d-adic$ decomposition,  \[\forall s\geq \log_d(C_2m)+1,u_s=0.\]
\end{proposition}

\begin{proof}
    Suppose that \[\exists s\geq \log_d(C_2m)+1,u_s\ne0.\]
    We write $R=\max{\{s\geq \log_d(C_2m)+1 : u_s\ne 0\}}$.
    By Lemma \ref{card_Pn}, we have $\abso{\ell-\ell^'}\leq C_2 m$.
    We know that $\abso{u_R}\geq 1$ because $u_R$ is an integer. Then, since $d^R\geq dC_2m$,
     \[dC_2m\abso{1+\sum_{j=0}^{R-1}\frac{u_j}{u_R}d^{j-R}}\leq\abso{u_Rd^R}\abso{\sum_{j=0}^R\frac{u_j}{u_R}d^{j-R}}=\abso{\ell-\ell^'}\leq C_2m.\]
    On the other hand \[\sum_{j=0}^{R-1}\abso{\frac{u_j}{u_R}}d^{j-R}=\frac{1}{d}\sum_{j=0}^{R-1}\abso{\frac{u_{R-1-j}}{u_R}}d^{-j}\leq \frac{d-1}{2d}\frac{1}{1-\frac{1}{d}}=\frac{1}{2}.\]
    Therefore, by reversed triangle inequality applied to $\abso{1+\sum_{j=0}^{R-1}\frac{u_j}{u_R}d^{j-R}}$, we get $\frac{dC_2 m}{2}\leq C_2 m$, a contradiction.
\end{proof}

The following lemma shows that the carry number is not propagated above low digits.

\begin{lemme}
    For $\alpha\not\in \check{W}_{k,n}$, for $\ell>\ell^'\in I_n^\alpha$, we have,\[\forall j\geq \log_d(C_2m)+1,a_{\ell,j}=a_{\ell^',j}.\]
\end{lemme}

\begin{proof}
    Let $\alpha\not\in \check{W}_{k,n}$ and $\ell>\ell^'\in I_n^\alpha$.
    By Proposition \ref{limit_d_adic}, we can pose $\xi=\ell-\ell^'=\sum_{j=0}^{p_n}u_jd^j$.
    Thus, \[\ell=\ell^'+\xi=\sum_{j=0}^{p_n}(a_{\ell^',j}+u_j)d^j+\sum_{j\geq p_n+1}a_{\ell^',j}d^j.\]
    By definition of $\check{W}_{k,n}$,
    \[\exists i\in \intervEnt{p_n+1}{q_n},a_{\ell^',i}< \frac{d-1}{2}.\]
    Then, \[\ell=\sum_{j=0}^{p_n}(a_{\ell^',j}+u_j)d^j+\sum_{j=p_n+1}^{i-1}(a_{\ell^',j})d^j+a_{\ell^',i}d^i+\sum_{j\geq i+1}a_{\ell^',j}d^j.\]
    Hence, \[\sum_{j=0}^{p_n}(a_{\ell^',j}+u_j)d^j+\sum_{j=p_n+1}^{i-1}a_{\ell^',j}d^j+a_{\ell^',i}d^i\leq d^{i}-1+d^i\left(\frac{d-1}{2}-1\right)=\frac{d-1}{2}d^i-1.\]
    Therefore, \[\forall j\geq i, a_{\ell,j}=a_{\ell^',j}.\]
\end{proof}

Thus, we have for $\alpha\not\in W_{k,n}\cup\check{W}_{k,n}$, \begin{align}
    \label{control_norme_return}\norme{{t_\ell^\alpha}^'-{t_{\ell^'}^\alpha}^'}_\infty=O(2\log_d(C_2m)).
\end{align}

We note $\sigma_\ell^\alpha=\sqrt{\mathbb{V}\left({t_\ell^\alpha}^'\right)}$.
Next lemma shows that these standard deviations are equivalent.
\begin{lemme}
\label{ecart_type_equiv}
    For $\alpha\not\in W_{k,n}\cup\check{W}_{k,n}$ and $\ell,\ell^'\in I_n^\alpha$, \[\frac{{\sigma_\ell^\alpha}}{{\sigma_{\ell^'}^\alpha}}=1+O\left(\frac{\log_d(C_2m)}{\sqrt{m}}\right).\]
\end{lemme}

\begin{proof}
Let     $\alpha\not\in W_{k,n}\cup\check{W}_{k,n}$ and let $\ell,\ell^'\in I_n^\alpha$.

Then, $\mathbb{V}({t_\ell^\alpha}^')=\mathbb{V}({t_\ell^\alpha}^'-{t_{\ell^'}^\alpha}^')+2Cov({t_\ell^\alpha}^'-{t_{\ell^'}^\alpha}^',{t_{\ell^'}^\alpha}^')+\mathbb{V}({t_{\ell^'}^\alpha}^').$

By Proposition \ref{minimize_var}, for some constant $C^'$, \[\mathbb{V}({t_{\ell^'}^\alpha}^')\geq C^'\frac{d-2}{(d-1)^2}b_{\ell^'}\geq C\left(\frac{d-2}{(d-1)^2}\right)m.\]
 By (\ref{control_norme_return}) and Cauchy-Schwarz inequality, we get, \[\abso{\frac{{\sigma_\ell^\alpha}}{{\sigma_{\ell^'}^\alpha}}-1}\leq\left(\frac{\sqrt{C_1\log_d^2(C_2m)+C_3\log_d(C_2m)}}{\sqrt{C\left(\frac{d-2}{(d-1)^2}\right)m}}\right).\]
 Therefore, \[\frac{{\sigma_\ell^\alpha}}{{\sigma_{\ell^'}^\alpha}}=1+O\left(\frac{\log_d(m)}{\sqrt{m}}\right).\]
\end{proof}

\subsection{Final estimate of self correlation}\label{self_cor}

To control the error when we change the standard deviations $\sigma_\ell^\alpha$ in the Proposition \ref{llt}, we use the next lemma.

\begin{lemme}
\label{control_simple_depend}
    For $A>0,(\beta_n)_{n\in\mathbb{N}}\in\left(\mathbb{R}^*_+\right)^\mathbb{N},(\varepsilon_n)_{n\in\mathbb{N}}\in\left(\mathbb{R}\setminus\{-1\}\right)^{\mathbb{N}},(\xi_{\ell,n})_{(\ell,n)\in \mathbb{Z}\times\mathbb{N}}\in \{0,1\}^{\mathbb{Z}\times\mathbb{N}}$ and $(\varepsilon_n^{(\ell)})_{(n,\ell)\in\mathbb{N}\times\mathbb{Z}}\in\left(\mathbb{R}^*_+\right)^{\mathbb{N}\times\mathbb{Z}}$ such that for all $\ell\in\mathbb{Z}$, all $n\in\mathbb{N},\abso{\varepsilon_n^{(\ell)}}\leq\abso{\varepsilon_n}$, we have
    \[\abso{\sum_{\ell\in \mathbb{Z}}\frac{\xi_{n,\ell}}{\left(1+\varepsilon_n^{(\ell)}\right)\beta_n\sqrt{2\pi}}e^{-\frac{(\ell A-n)^2}{2\left(1+\varepsilon_n\right)^2\beta_n^2}}-\frac{\xi_{n,\ell}}{\sqrt{2\pi}\beta_n}\sum_{\ell\in\mathbb{Z}}e^{-\frac{(\ell A-n)^2}{2\beta_n^2}}}\leq \frac{\abso{\varepsilon_n}}{A}\left(\frac{2}{\sqrt{2\pi e}}+\frac{1}{\beta_n}\right).\]
\end{lemme}

\begin{proof}
    We know that, $\forall x\in\mathbb{R},\frac{d^2}{dx^2}e^{-\frac{x^2}{2}}=(x^2-1)e^{-\frac{x^2}{2}}$,
    For the second derivative, we have \[\forall x\in\mathbb{R},\left((x> 1\ or\ x<-1) \iff \frac{d^2}{dx^2}e^{-\frac{x^2}{2}}>0\right).\]
    The derivative vanishes to $\pm\infty$, i.e $(x^2-1)e^{-\frac{x^2}{2}}\xrightarrow[x\to\pm \infty]{}0$.
    Then, \begin{align}\label{deriv_control}\forall x\in\mathbb{R},\abso{\frac{d}{dx}e^{-\frac{x^2}{2}}}\leq \frac{1}{\sqrt{e}}.\end{align}
    We know also that $\frac{d^3}{dx^3}e^{-\frac{x^2}{2}}=x(3-x^2)e^{-\frac{x^2}{2}}$ and we get that \begin{equation}
        \label{ineq_deriv_3}\forall x\in\mathbb{R},\abso{\frac{d^2}{dx^2}e^{-\frac{x^2}{2}}}\leq 1.
    \end{equation}
    Consider for $x\in\mathbb{R}$, \[\begin{array}{ll}\abso{\frac{\xi_{n,\ell}}{(1+\varepsilon_n^{(\ell)})\beta_n}e^{-\frac{x^2}{2(1+\varepsilon_n^{(\ell)})^2\beta_n^2}}-\frac{\xi_{n,\ell}}{\beta_n}e^{-\frac{x^2}{2\beta_n^2}}}&\leq \abso{\frac{1}{(1+\varepsilon_n^{(\ell)})\beta_n}e^{-\frac{x^2}{2(1+\varepsilon_n^{(\ell)})^2\beta_n^2}}-\frac{1}{\beta_n}e^{-\frac{x^2}{2\beta_n^2}}}\\
    &\leq \abso{\frac{1}{(1+\varepsilon_n^{(\ell)})\beta_n}e^{-\frac{x^2}{2(1+\varepsilon_n^{(\ell)})^2\beta_n^2}}-\frac{1}{(1+\varepsilon_n^{(\ell)})\beta_n}e^{-\frac{x^2}{2\beta_n^2}}}\\
    &+\abso{\frac{1}{(1+\varepsilon_n^{(\ell)})\beta_n}e^{-\frac{x^2}{2\beta_n^2}}-\frac{1}{\beta_n}e^{-\frac{x^2}{2\beta_n^2}}}\\
    &\leq \frac{1}{(1+\varepsilon_n^{(\ell)})^2\beta_n^2\sqrt{e}}\abso{x}\abso{\varepsilon_n^{(\ell)}}\\
    &+\abso{\frac{1}{(1+\varepsilon_n^{(\ell)})\beta_n}e^{-\frac{x^2}{2\beta_n^2}}-\frac{1}{\beta_n}e^{-\frac{x^2}{2\beta_n^2}}}\ by\ \eqref{deriv_control}\\
    &\leq \frac{e^{-\frac{x^2}{2\beta_n^2}}}{(1+\varepsilon_n^{(\ell)})^2\beta_n^2\sqrt{e}}\abso{x}\abso{\varepsilon_n^{(\ell)}}+\frac{e^{-\frac{x^2}{2\beta_n^2}}}{\beta_n\abso{1+\varepsilon_n^{(\ell)}}}\abso{\varepsilon_n^{(\ell)}}\\
    &\leq \frac{\abso{\varepsilon_n^{(\ell)}}e^{-\frac{x^2}{2\beta_n^2}}}{\beta_n\abso{1+\varepsilon_n^{(\ell)}}}\left(\frac{\abso{x}}{\beta_n\abso{1+\varepsilon_n^{(\ell)}}\sqrt{e}}+1\right)\\
    &\leq \frac{\abso{\varepsilon_n^{(\ell)}}e^{-\frac{x^2}{2\beta_n^2}}}{\beta_n(1+\varepsilon_n^{(\ell)})}\left(\frac{\abso{x}}{\beta_n(1+\varepsilon_n^{(\ell)})\sqrt{e}}+1\right)\ because\ \varepsilon_n^{(\ell)}>0.\end{array}\]
    Therefore, by serie-integrals comparaison and \eqref{ineq_deriv_3}, \[\abso{\sum_{\ell\in \mathbb{Z}}\frac{\xi_{n,\ell}}{\left(1+\varepsilon_n\right)\beta_n\sqrt{2\pi}}e^{-\frac{(\ell A-n)^2}{2\left(1+\varepsilon_n\right)^2\beta_n^2}}-\frac{\xi_{n,\ell}}{\sqrt{2\pi}\beta_n}\sum_{\ell\in\mathbb{Z}}e^{-\frac{(\ell A-n)^2}{2\beta_n^2}}}\leq 2\frac{\abso{\varepsilon_n}}{A}\left(\frac{1}{\sqrt{2\pi e}}+\frac{1}{\beta_n}\right).\]
\end{proof}

We are now ready to prove the self correlation estimate.

\begin{proof}[Proof of Proposition \ref{first_keplerian_shear}]

Consider $\sum_{\ell\in\mathbb{N}}d_\ell^\alpha(n)=\sum_{\ell\in P_n^\alpha}d_\ell^\alpha(n)$.
We recall that $S_m={t_\ell^\alpha}^'-\mathbb{E}({t_\ell^\alpha}^')$.
We can cut the set $I_n^\alpha$ in two parts $A_n^\alpha$ and $B_n^\alpha$  as follows
\[A_n^\alpha:=\{\ell\in I_n^\alpha : \abso{\mathbb{E}({t_\ell^\alpha}^')-n}\leq m^{\frac{3}{4}}\}\ and\ B_n^\alpha:={\left(A_n^\alpha\right)}^c\cap I_n^\alpha.\]
Thus, \[\sum_{\ell\in\mathbb{N}}d_\ell^\alpha(n)=\sum_{\ell\in A_n^\alpha}{d_\ell^\alpha}^'(n)+\sum_{\ell\in B_n^\alpha}{d_\ell^\alpha}^'(n).\]

Notice that if ${t_\ell^\alpha}^'(x)=n$ and $\abso{\mathbb{E}({t_\ell^\alpha}^')-n}>m^{\frac{3}{4}}$ then $\abso{\mathbb{E}({t_\ell^\alpha}^')-{t_\ell^\alpha}^'(x)}>m^{\frac{3}{4}}$.
Thus by Proposition \ref{control_chernov}, \[\sum_{\ell\in B_n^\alpha}{d_\ell^\alpha}^'(n)\leq 2C\# I_n^\alpha e^{-m^{\frac{1}{4}}}, C>0.\]
By Lemma \ref{card_Pn},
\[\# I_n^\alpha\leq C_4m, C_4>0.\]
Then, \[\sum_{\ell\in B_n^\alpha}{d_\ell^\alpha}^'(n)\leq 2CC_4m e^{-m^{\frac{1}{4}}}.\]
Let $\ell_{min}^{(n)}$ be the minimum of $I_n^\alpha$ and set $\hat\sigma_n:=\sigma_{\ell_{min}^{(n)}}$.
By Kac's lemma, $\mathbb{E}({t_\ell^\alpha}^')=\ell d^k(\alpha+1)$, therefore we have that $\ell\in I_n^\alpha\iff \abso{\mathbb{E}({t_\ell^\alpha}^')-n}\leq C_2 m$, and we know by \eqref{I_n} that $I_n^\alpha$ is an interval.
Thus, $$\frac{1}{\hat\sigma_n\sqrt{2\pi}}\sum_{\ell\not\in I_n^\alpha}e^{-\frac{(\mathbb{E}({t_\ell^\alpha}^')-n)^2}{2\hat\sigma_n^2}}=\frac{1}{\hat\sigma_n\sqrt{2\pi}}\sum_{\abso{\mathbb{E}({t_\ell^\alpha}^')-n}>C_2m}e^{-\frac{(\mathbb{E}({t_\ell^\alpha}^')-n)^2}{2\hat\sigma_n^2}}\leq \frac{e^{-m^{\frac{1}{4}}}}{\sqrt{2m\pi}}\sqrt{2}$$ 
by the convergence of the serie.
We get, \begin{equation}\label{eq_interval}\frac{1}{\hat\sigma_n\sqrt{2\pi}}\sum_{\ell\not\in I_n^\alpha}e^{-\frac{(\mathbb{E}({t_\ell^\alpha}^')-n)^2}{2\hat\sigma_n^2}}=O\left(\frac{e^{-\frac{m}{4}}}{\sqrt{m}}\right).\end{equation}
By Proposition \ref{llt} and applying Lemma \ref{ecart_type_equiv} on $\sigma({t_\ell^\alpha}^')$  and Lemma \ref{control_simple_depend} on $\sum_{\ell\not\in I_n^\alpha}\frac{1}{\sigma_{\ell}^\alpha\sqrt{2\pi}}e^{-\frac{(\mathbb{E}({t_\ell^\alpha}^')-n)^2}{2\left(\sigma_{\ell}^\alpha\right)^2}}$, we get \[\sum_{\ell\in\mathbb{N}}{d_\ell^\alpha}^'(n)=\frac{1}{\hat\sigma_n\sqrt{2\pi}}\sum_{\ell\in I_n^\alpha}e^{-\frac{(\ell d^k(\alpha+1)-n)^2}{2\hat\sigma_n^2}}+O\left(\frac{\sqrt{\log_d(m)}}{m^{\frac{1}{4}}}\right).\]
Hence by (\ref{eq_interval}),\[\sum_{\ell\in\mathbb{N}}{d_\ell^\alpha}^'(n)=\frac{1}{\hat\sigma_n\sqrt{2\pi}}\sum_{\ell\in\mathbb{N}}e^{-\frac{(\ell d^k(\alpha+1)-n)^2}{2\hat\sigma_n^2}}+O\left(\frac{\sqrt{\log_d(m)}}{m^{\frac{1}{4}}}\right).\]

Just for the next equation, we pose $A=d^k(\alpha+1)$.
By the serie-integral comparaison,
\begin{equation}
    \label{serie_int_gauss}
    \int_0^{+\infty}e^{-\frac{(Ax-n)^2}{2\hat\sigma_n^2}}dx-\int_{\lfloor \frac{n}{A}\rfloor-1}^{\lfloor \frac{n}{A}\rfloor+1}e^{-\frac{(Ax-n)^2}{2\hat\sigma_n^2}}dx+e^{-\frac{n^2}{2\hat\sigma_n^2}}\leq\sum_{\ell\in\mathbb{N}}e^{-\frac{(\ell A-n)^2}{2\hat\sigma_n^2}},
\end{equation}
and,
\begin{equation}
    \label{serie_int_gauss_2}
    \sum_{\ell\in\mathbb{N}}e^{-\frac{(\ell A-n)^2}{2\hat\sigma_n^2}}\leq \int_0^{+\infty}e^{-\frac{(Ax-n)^2}{2\hat\sigma_n^2}}dx-\int_{\lfloor \frac{n}{A}\rfloor-1}^{\lfloor \frac{n}{A}\rfloor+1}e^{-\frac{(Ax-n)^2}{2\hat\sigma_n^2}}dx+e^{-\frac{(A(\lfloor\frac{n}{A}\rfloor+1)-n)^2}{2\hat\sigma_n^2}}+e^{-\frac{(A\lfloor\frac{n}{A}\rfloor-n)^2}{2\hat\sigma_n^2}}.
\end{equation}
Thus, by \eqref{serie_int_gauss} and \eqref{serie_int_gauss_2},
\[\frac{1}{\hat\sigma_n\sqrt{2\pi}}\sum_{\ell\in\mathbb{N}}e^{-\frac{(\ell d^k(\alpha+1)-n)^2}{2\hat\sigma_n^2}}=\frac{1}{\hat\sigma_n\sqrt{2\pi}}\int_0^{+\infty}e^{-\frac{(d^k(\alpha+1)x-n)^2}{2\hat\sigma_n^2}}dx+O\left(\frac{1}{\hat\sigma_n}\right).\]
Finally, by change of variables,
\[\sum_{\ell\in\mathbb{N}}{d_\ell^\alpha}^'(n)=\frac{1}{d^{k}(\alpha+1)}+O\left(\frac{\sqrt{\log_d(m)}}{m^{\frac{1}{4}}}\right)=\lambda_\alpha(A_k)+O\left(\frac{\sqrt{\log_d(\log_d(n))}}{(\log_d(n))^{\frac{1}{4}}}\right).\]
\end{proof}

\subsection{Desynchronization of the exceptional sets for $A_k$}
The goal of this subsection is to prove the convergence of the whole auto-correlation linked to $A_k$.
\begin{proposition} 
\label{kep_shear_one}
For all $B\in \mathcal{B}(\{1,2\}^\mathbb{N})$, \[\intInd{\Omega}{\mathds{1}_{(B\times A_k)\cap T^{-n}(B\times A_k)}}{\mu}\xrightarrow[n\to+\infty]{}\intInd{\{1,2\}^\mathbb{N}}{\mathds{1}_{B}(\alpha)\left(\lambda_\alpha(A_k)\right)^2}{\mathcal{H}(\alpha)}.\]
\end{proposition}

\begin{proof}
    Let $B\in\mathcal{B}(\{1,2\}^\mathbb{N})$ and let $k\in\mathbb{N}$.
    
    Considering exceptional sets defined in (\ref{first_exceptional_set}), (\ref{second_exceptional_set}) and (\ref{M_m_define}), set $Z_{k,n}:=W_{k,n}\cup\check{W}_{k,n}\cup M_{\lfloor\log_d(n)\rfloor}$.
    We note \[\mathbb{E}(\mathds{1}_{(B\times A_k)\cap T^{-n}(B\times A_k)})=\intInd{\Omega}{\mathds{1}_{(B\times A_k)\cap T^{-n}(B\times A_k)}}{\mathcal{H}\otimes(\mathbb{L}eb_\alpha)_{\alpha\in\{1,2\}^\mathbb{N}}}.\]
    We have that \[\mathbb{E}(\mathds{1}_{(B\times A_k)\cap T^{-n}(B\times A_k)})=\mathbb{E}(\mathds{1}_{((Z_{k,n}\cap B)\times A_k)\cap T^{-n}(B\times A_k)})+\mathbb{E}(\mathds{1}_{((Z_{k,n}^c\cap B)\times A_k)\cap T^{-n}(B\times A_k)}).\]
    Note that, \[\mathbb{E}(\mathds{1}_{((Z_{k,n}\cap B)\times A_k)\cap T^{-n}(B\times A_k)})=O(\mathcal{H}(Z_{k,n})).\]
    By Propositions \ref{vanish_ex_set}, \ref{second_vanish_ex_set} and Lemma \ref{control_Mm}, \[\mathcal{H}(Z_{k,n})\xrightarrow[n\to+\infty]{}0.\]
    We get by Proposition \ref{first_keplerian_shear}, \[\mathbb{E}(\mathds{1}_{((Z_{k,n}^c\cap B)\times A_k)\cap T^{-n}(B\times A_k)})\xrightarrow[n\to+\infty]{}\intInd{\{1,2\}^\mathbb{N}}{\mathds{1}_{B}(\alpha)\left(\lambda_\alpha(A_k)\right)^2}{\mathcal{H}(\alpha)}.\]
\end{proof}

\subsection{From self correlations to keplerian shear}

Now, we want to extend the preceding result for all Borel set of $\Omega$.
It's enough to prove the convergence to 0 of autocorrelations, we prove that in the next lemma.

\begin{lemme}
    Let $(\Omega,\mathcal{T},\mu,T)$ be a dynamical system.
    Let $f\in \mathbb{L}_\mu^2(\Omega)$ such that \[\mathbb{E}(Cov_n(f,f\vert \mathcal{I}))\xrightarrow[n\to+\infty]{}0.\]
    Then, \[\forall g\in  \mathbb{L}_\mu^2(\Omega),\mathbb{E}(Cov_n(f,g\vert \mathcal{I}))\xrightarrow[n\to+\infty]{}0.\]
    \label{autocorrel}
\end{lemme}
\begin{proof}
    Let $f\in \mathbb{L}_\mu^2(\Omega)$ such that \[\mathbb{E}(Cov_n(f,f\vert \mathcal{I}))\xrightarrow[n\to+\infty]{}0.\]
    Let \[H_f :=\overline{Vect\left(\left\{\mathbb{E}(f\vert \mathcal{I})\right\}\cup\{f\circ T^{n} : n\in\mathbb{N}\}\right)}\]
    and let \[F_f := \{g\in \mathbb{L}^2_\mu(\Omega) : \mathbb{E}(Cov_n(f,g\vert \mathcal{I}))\xrightarrow[n\to+\infty]{} 0\}.\]
    We prove easily that $F_f$ is closed.
    It follow from the assumption that $H_f\subset F_f$.
    Now, we prove that $H_f^{\perp}\subset F_f$. Let $g\in H_f^{\perp}$. Then, \[\forall n\in\mathbb{N},\intInd{\Omega}{(f\circ T^n)g}{\mu}=0=\intInd{\Omega}{\mathbb{E}(f\vert\mathcal{I})g}{\mu}.\]
    This implies \[\forall n\in\mathbb{N},\mathbb{E}(Cov_n(f,g\vert \mathcal{I}))=0,\]
    proving the claim.
    Then, $\mathbb{L}_\mu^2(\Omega)=F_f$.
\end{proof}

To extend the convergence proved in Proposition \ref{kep_shear_one} for every measurable set of $\Omega$, we will use the Dynkin class lemma. We prove a general result which says that the set of Borel sets verifying the convergence to 0 of autocorrelations is a Dynkin class.

\begin{lemme}
    \label{DynkinKepler}
    Let $(\Omega,\mathcal{T},\mu,T)$ be a dynamical system.
    Let $\mathcal{S}\subset\mathcal{T}$ such that \[\forall A\in \mathcal{S},\mathbb{E}(Cov_n(\mathds{1}_A,\mathds{1}_A\vert \mathcal{I}))\xrightarrow[n\to+\infty]{}0.\]
    We note $\mathscr{C}(\mathcal{S})$ the Dynkin class generated by $\mathcal{S}$.
    Then, \[\forall A\in \mathscr{C}(\mathcal{S}),\mathbb{E}(Cov_n(\mathds{1}_A,\mathds{1}_A\vert \mathcal{I}))\xrightarrow[n\to+\infty]{}0.\]
\end{lemme}

\begin{proof}
    Let $\mathcal{R} := \left\{A\in\mathcal{T} : \mathbb{E}(Cov_n(\mathds{1}_A,\mathds{1}_A\vert \mathcal{I}))\xrightarrow[n\to+\infty]{}0\right\}$.
    We prove that $\mathcal{R}$ is a Dynkin class.
    Easyly, we get $\Omega\in \mathcal{R}$. Let $A,B\in\mathcal{R}$ such that $A\subset B$.
    Then,\[\begin{array}{ll}
         \mathbb{E}(Cov_n(\mathds{1}_{B\setminus A},\mathds{1}_{B\setminus A}\vert \mathcal{I}))&=\mathbb{E}(Cov_n(\mathds{1}_{B}-\mathds{1}_{A},\mathds{1}_{B}-\mathds{1}_{A}\vert\mathcal{I}))  \\
         &=\mathbb{E}(Cov_n(\mathds{1}_{B},\mathds{1}_{B}\vert\mathcal{I}))+\mathbb{E}(Cov_n(\mathds{1}_{A},\mathds{1}_{A}\vert\mathcal{I}))-2\mathbb{E}(Cov_n(\mathds{1}_{A},\mathds{1}_{B}\vert\mathcal{I}))
    \end{array}\]
Then, by Lemma \ref{autocorrel},\[\mathbb{E}(Cov_n(\mathds{1}_{B\setminus A},\mathds{1}_{B\setminus A}\vert \mathcal{I}))\xrightarrow[n\to+\infty]{}0.\]

Let $(A_n)_{n\in\mathbb{N}}\in\mathcal{R}^\mathbb{N}$ such that \[\forall n\in\mathbb{N},A_n\subset A_{n+1}.\]
Let $A=\union{n\in\mathbb{N}}{A_n}$.
We prove that $\mathbb{E}(Cov_n(\mathds{1}_A,\mathds{1}_A\vert\mathcal{I}))\xrightarrow[n\to+\infty]{}0$.
For $m,n\in\mathbb{N}$ we have \[\mathbb{E}(Cov_n(\mathds{1}_A,\mathds{1}_A\vert\mathcal{I}))=\mathbb{E}(Cov_n(\mathds{1}_{A\setminus A_m},\mathds{1}_{A\setminus A_m}\vert\mathcal{I}))+2\mathbb{E}(Cov_n(\mathds{1}_{A\setminus A_m},\mathds{1}_{A_m}\vert\mathcal{I}))+\mathbb{E}(Cov_n(\mathds{1}_{A_m},\mathds{1}_{A_m}\vert\mathcal{I})).\]
Thus, for $m,n\in\mathbb{N}$, $\abso{\mathbb{E}(Cov_n(\mathds{1}_A,\mathds{1}_A\vert\mathcal{I}))}\leq 2(\mu(A\setminus A_m)+\left(\mu(A\setminus A_m)\right)^{\frac{1}{2}}+\abso{\mathbb{E}(Cov_n(\mathds{1}_{A_m},\mathds{1}_{A_m}\vert\mathcal{I}))}).$
Therefore, there exists $m_0\in\mathbb{N}$ such that $(\mu(A\setminus A_{m_0})+(\mu(A\setminus A_{m_0}))^{\frac{1}{2}}<\frac{\varepsilon}{4})$,
and there exists $n_0\in\mathbb{N}$ such that, \[\forall n\geq n_0,\abso{\mathbb{E}(Cov_n(\mathds{1}_{A_{m_0}},\mathds{1}_{A_{m_0}}\vert\mathcal{I}))}<\frac{\varepsilon}{4}.\]
Hence,\[\forall n\geq n_0,\mathbb{E}(Cov_n(\mathds{1}_A,\mathds{1}_A\vert\mathcal{I}))<\varepsilon.\]
Thus $A\in\mathcal{R}$.
We get that $\mathcal{R}$ is a Dynkin class.
We know that $\mathcal{S}\subset \mathcal{R}$.
Then, \[\mathscr{C}(\mathcal{S})\subset \mathscr{C}(\mathcal{R})=\mathcal{R}.\]
\end{proof}

And now, we show how we use the Dynkin class lemma to prove the keplerian shear.

\begin{proposition}[Keplerian shear on $\pi-$system]
    \label{kepler_dynkin}
    Let $(\Omega,\mathcal{T},\mu,T)$ be a dynamical system.
    Let $\mathcal{S}\subset\mathcal{T}$ a $\pi-$system such that $\sigma(\mathcal{S})=\mathcal{T}$ and \[\forall A\in \mathcal{S},\mathbb{E}(Cov_n(\mathds{1}_A,\mathds{1}_A\vert\mathcal{I}))\xrightarrow[n\to+\infty]{}0.\]
    Then, \[\forall f\in \mathbb{L}^2_\mu(\Omega),f\circ T^n\xrightharpoonup[n\to+\infty]{}\mathbb{E}(f\vert \mathcal{I}).\]
\end{proposition}
\begin{proof}
    By Lemma \ref{DynkinKepler}, \[\forall A\in \mathscr{C}(\mathcal{S}),\mathbb{E}(Cov_n(\mathds{1}_A,\mathds{1}_A\vert \mathcal{I}))\xrightarrow[n\to+\infty]{}0.\]
    Since $\mathcal{S}$ is a $\pi-$system, by Dynkin class lemma,\[\mathscr{C}(\mathcal{S})=\sigma(\mathcal{S}).\]
    Furthermore $\sigma(\mathcal{S})=\mathcal{T},$
    hence, \[\forall A\in\mathcal{T},\mathbb{E}(Cov_n(\mathds{1}_A,\mathds{1}_A\vert \mathcal{I}))\xrightarrow[n\to+\infty]{}0.\]
    By Lemma \ref{autocorrel}, this gives,\[\forall A,B\in\mathcal{T},\mathbb{E}(Cov_n(\mathds{1}_A,\mathds{1}_B\vert \mathcal{I}))\xrightarrow[n\to+\infty]{}0.\]
    Thus, \[\forall f\in \mathbb{L}^2_\mu(\Omega),f\circ T^n\xrightharpoonup[n\to+\infty]{}\mathbb{E}(f\vert \mathcal{I}).\]
\end{proof}
\subsection{Convergence on cylinders}

Now, we want to extend this convergence for all Borel subsets of $\Omega$.
It suffices to prove the convergence to 0 of conditional autocorrelations, the object of next lemma.
Now, we prove that the convergence of autocorrelations is verified on $\alpha-$cylinders, which form a $\pi-$system. This allows to use Proposition~\ref{kepler_dynkin}.
\begin{lemme}
    \label{cylinder_shear}
    Let $k,m\in\mathbb{N}$, and let $u\in\{1,2\}^{\intervEnt{0}{m}}$.
    Then, \[\mathbb{E}(Cov_n(\mathds{1}_{[u]\times A_k},\mathds{1}_{[u]\times A_k}\vert \mathcal{I}))\xrightarrow[n\to+\infty]{}0.\]
\end{lemme}
\begin{proof}
    Let $k,m\in\mathbb{N}$, and let $u\in\{1,2\}^{\intervEnt{0}{m}}$.
    Taking $B=[u]$, using Proposition \ref{kep_shear_one}, we get the result.
\end{proof}

Therefore, we can generate every $d-$adic subdivisions of $[0,1[$ with transformation of $\alpha-$cylinder and $A_k$. 

\begin{lemme}
    \label{Shear_translation}
    Let $k\in\mathbb{N}, m\geq k$ and let $u\in\{1,2\}^{\intervEnt{0}{m}}$.
    Let $j\in \intervEnt{1}{d^k}$.
    
    Then, there exists $s\in\mathbb{N}$ such that $s\leq h_k^{(u)}$ and \[[u]\times\left[\frac{j-1}{d^k},\frac{j}{d^k}\right[=T^s([u]\times A_k)\]
    where $h_k^{(u)}$ is the common value of $h_k^\alpha$ for $\alpha\in [u]$.
\end{lemme}
\begin{proof}
     Let $k\in\mathbb{N}, m\geq k$ and let $u\in\{1,2\}^{\intervEnt{0}{m}}$.
    Let $(\alpha,x)\in [u]\times A_k$.
    We write $j-1=\sum_{i=0}^{k-1}v_id^{i}$ with $v_i\in\intervEnt{0}{d-1}$.
    We take \[s=\sum_{i=0}^{k-1}(v_{i}h_{i}^{\alpha}+\alpha_i\mathds{1}_{\{d-1\}}(v_i)).\]
    By (\ref{equa_trans}) and (\ref{equa_final_step}), \[T^s([u]\times A_k)=[u]\times\left[\frac{j-1}{d^k},\frac{j}{d^k}\right[.\]
\end{proof}

Now, we can also generate the spacers, using the next lemma.
\begin{lemme}
    \label{Shear_spacer}
    Let $m\in\mathbb{N}$, and let $u\in\{1,2\}^{\intervEnt{0}{m}}$.
    Let $j\in \intervEnt{0}{m-1}$ and $l\in \intervEnt{0}{u_j-1},m\geq k\geq j+1$ and $q\in \intervEnt{1}{u_jd^{k-(j+1)}}$.
    Thus, there exists $s\in\mathbb{N}$ such that $s\leq h_k^{(u)}$ and \[[u]\times\left(\left\{1+\sum_{i=0}^{j-1}u_id^{-(i+1)}+ld^{-(j+1)}\right\}+\left[\frac{q-1}{d^{k}},\frac{q}{d^{k}}\right[\right)=T^s([u]\times A_k).\]
\end{lemme}
\begin{proof}
    The proof is analogous to the proof of Lemma \ref{Shear_translation}.
\end{proof}
Therefore, we can build the $\pi-$system for the convergence with cylinders.
We define \[\mathcal{S}_1:=\left\{[u]\times\left[\frac{j-1}{d^k},\frac{j}{d^k}\right[: m\in\mathbb{N},\ \ m\geq k,\ \ u\in\{1,2\}^{\intervEnt{0}{m}}\ and\ j\in \intervEnt{1}{d^k}\right\}\cup\{\emptyset\}.\]
Let \[\mathcal{S}_2:=\left\{\begin{array}{ll}[u]\times\left(\left\{1+\sum_{i=0}^{j-1}u_id^{-(i+1)}+ld^{-(j+1)}\right\}+\left[\frac{q-1}{d^{k}},\frac{q}{d^{k}}\right[\right): \\
m\in\mathbb{N},\ \ u\in\{1,2\}^{\intervEnt{0}{m}},\ \ j\in \intervEnt{0}{m-1},\\
l\in \intervEnt{0}{u_j-1}\ and\ m\geq k\geq j+1,\ q\in \intervEnt{1}{u_jd^{k-(j+1)}}\end{array}\right\}\cup\{\emptyset\}.\]

Finally, the union of the two previous set is a $\pi-$system.
\begin{lemme}
\label{pi_system_park}
$\mathcal{S}_1\cup\mathcal{S}_2$ is stable by intersection :
    \[\forall (A,B)\in (\mathcal{S}_1\cup\mathcal{S}_2)^2,A\cap B\in \mathcal{S}_1\cup\mathcal{S}_2.\]
\end{lemme}

Furthermore, on $\mathcal{S}_1$ and $\mathcal{S}_2$, we have the convergence to 0 of correlations ensured.

\begin{lemme}\label{conv_S1S2}
    For all $A\in \mathcal{S}_1\cup\mathcal{S}_2$, \[\mathbb{E}(Cov_n(\mathds{1}_{A},\mathds{1}_{A}\vert\mathcal{I}))\xrightarrow[n\to+\infty]{}0.\]
\end{lemme}

\begin{proof}
    Let $A\in  \mathcal{S}_1$.
    By Lemma \ref{Shear_translation}, there exist $m\in\mathbb{N},\ u\in\{1,2\}^{\intervEnt{0}{m}},\ j\in \intervEnt{1}{d^k}$ and $s\in\mathbb{N}$ such that $s\leq h_k^{(u)}$ and \[A=[u]\times\left[\frac{j-1}{d^k},\frac{j}{d^k}\right[=T^s([u]\times A_k).\]
    Observe that for any $\alpha\in [u]$, since $s\leq h_k^{\alpha}$,\[T_\alpha^{-s}(T_\alpha^s(A_k)\cap T_\alpha^{-n}(T_\alpha^s(A_k)))=A_k\cap T_\alpha^{-n}(A_k).\]
    Therefore,  \[\mathbb{E}(Cov_n(\mathds{1}_{A},\mathds{1}_{A}\vert\mathcal{I}))\xrightarrow[n\to+\infty]{}0.\]
    Similarly, we get the same with $A\in\mathcal{S}_2$.
\end{proof}

Now, we have all elements to prove the main result of the article, the keplerian shear of the Chacon transformation.

\begin{proof}[Proof of Theorem \ref{theorem_Chacon}]
    With elements of $\mathcal{S}_1\cup \mathcal{S}_2$, by Lemma \ref{Shear_translation} and Lemma \ref{Shear_spacer}, we know that we can generate the sets of the form $[u]\times \left[0, a\right[$ with $[u]$ a cylinder of $\mathcal{C}$ and $a=\sum_{j=0}^{m}a_jd^{-(j+1)}$ with $a_j\in \intervEnt{0}{d-2}$.
    Then, $\sigma(\mathcal{S}_1\cup\mathcal{S}_2)=\mathcal{B}\left(\Omega\right)$.
    By Lemma \ref{pi_system_park}, we know that $\mathcal{S}_1\cup\mathcal{S}_2$ is a $\pi-$system.
    Finally, by Lemma \ref{conv_S1S2} and by Proposition \ref{kepler_dynkin}, \[\forall f\in \mathbb{L}^2_\mu\left(\Omega\right),f\circ T^n \xrightharpoonup[n\to+\infty]{}\mathbb{E}_\mu(f\vert\mathcal{I}).\]
\end{proof}

\end{document}